\documentclass[11pt]{amsart}    

\usepackage[pagebackref=true,colorlinks,linkcolor=blue,citecolor=magenta]{hyperref}
\usepackage{amsmath,amssymb,latexsym} 
\usepackage{graphicx}\usepackage{centernot}

\usepackage{fullpage}
\usepackage[square, numbers]{natbib}   
\usepackage{comment}
\usepackage{enumitem}

\newtheorem{theorem}{Theorem} 

\newtheorem{alphtheorem}{Theorem}

\newtheorem{alphlemma}{Lemma}

\newtheorem{corollary}{Corollary}
\newtheorem{proposition}{Proposition}

\newtheorem{lemma}{Lemma}
\theoremstyle{definition}

\newtheorem{question}{Question}

\theoremstyle{remark}

\def\Z{\mathbb{Z}}

\def\F{\mathcal{F}}

\def\zero{\boldsymbol{0}}

\def\ds{\displaystyle}
\def\KG{\operatorname{KG}}
\def\cd{\operatorname{cd}}

\def\alt{\operatorname{alt}}

\def\sd{\operatorname{sd}}

\title{Colorful Subhypergraphs in Uniform Hypergraphs}
\author{Meysam Alishahi}
\address{M. Alishahi, 
School of Mathematical Sciences,
Shahrood University of Technology, Shahrood, Iran}
\email{meysam\_alishahi@shahroodut.ac.ir}

\begin{document}
\maketitle

\begin{abstract}
There are several topological results ensuring the existence of a large complete bipartite subgraph in any properly colored graph satisfying some special topological regularity conditions. In view of $\mathbb{Z}_p$-Tucker lemma, 
Alishahi and Hajiabolhassan [{\it On the chromatic number of general Kneser hypergraphs, Journal of Combinatorial Theory, Series B, 2015}] introduced a lower bound for the chromatic number of 
Kneser hypergraphs ${\rm KG}^r({\mathcal H})$. 
Next, Meunier [{\it Colorful subhypergraphs in
Kneser hypergraphs, The Electronic Journal of Combinatorics, 2014}] improved their result by proving that any properly colored general Kneser hypergraph  ${\rm KG}^r({\mathcal H})$ contains a large colorful $r$-partite subhypergraph provided that $r$ is prime. 
In this paper, we give some new generalizations of $\mathbb{Z}_p$-Tucker lemma.
Hence, improving Meunier's result in some aspects.  Some new lower bounds for the chromatic number and local chromatic number of uniform hypergraphs are presented as well. \\

\end{abstract}

\section{\bf Introduction}
\subsection{{\bf Background and Motivations}}
In 1955, Kneser~\cite{MR0068536} posed a conjecture about the chromatic number 
of Kneser graphs. In 1978, Lov{\'a}sz~\cite{MR514625} proved this conjecture by 
using algebraic topology.
The Lov{\'a}sz's proof marked the beginning of the history of topological combinatorics.
Nowadays,  it is an active stream of research to study the coloring properties of graphs  
by using algebraic topology. There are several lower bounds for the chromatic number 
of graphs  related to the indices of some topological spaces defined based on the 
structure of graphs. 
However, for hypergraphs, there are a few such lower bounds, see~\cite{2013arXiv1302.5394A,MR953021,Iriye20131333,MR1081939,MR2279672}.

A {\it hypergraph} ${\mathcal H}$ is a pair $(V({\mathcal H}),E({\mathcal H}))$, where $V({\mathcal H})$ is a finite set, called the vertex set 
of ${\mathcal H}$, and $E({\mathcal H})$ is a family of nonempty subsets of $V({\mathcal H})$, called the edge set of ${\mathcal H}$. 
Throughout the paper, by a nonempty hypergraph, we mean a hypergraph 
with at least one edge. 
If any edge $e\in E({\mathcal H})$ has the cardinality $r$, then the hypergraph 
${\mathcal H}$ is called  {\it $r$-uniform.} For a set $U\subseteq V({\mathcal H})$, the {\it induced subhypergraph on $U$,} denoted ${\mathcal H}[U]$, is a hypergraph with the vertex set $U$ and the edge set 
$\{e\in E({\mathcal H}):\; e\subseteq U\}$.
Throughout the paper, by a {\it graph,} we mean a $2$-uniform hypergraph. 
Let $r\geq 2$ be a positive integer and $q\geq r$ be an integer.
An $r$-uniform hypergraph ${\mathcal H}$ is called {\it $q$-partite} with parts $V_1,\ldots,V_q$ if 
\begin{itemize}
\item $V({\mathcal H})=\displaystyle\bigcup_{i=1}^q V_i$ and
\item each edge of ${\mathcal H}$ intersects each part $V_i$ in at most one vertex.
\end{itemize}
If  ${\mathcal H}$ contains all possible edges, then we call it a
{\it complete $r$-uniform $q$-partite hypergraph.}
Also, we say the hypergraph  ${\mathcal H}$ is {\it balanced}  if the values of $|V_j|$ for 
$j =1,\ldots,q$ differ by at most one, i.e., $|V_i|-|V_j|\leq 1$ for each $i,j\in[q]$.

Let ${\mathcal H}$ be an $r$-uniform hypergraph and $U_1,\ldots, U_q$ be $q$ pairwise disjoint 
subsets of $V({\mathcal H})$.
The hypergraph ${\mathcal H}[U_1,\ldots, U_q]$ is a subhypergraph of ${\mathcal H}$ with the vertex set $\displaystyle\bigcup_{i=1}^q U_i$ and the edge set $$E({\mathcal H}[U_1,\ldots, U_q])=\left\{e\in E({\mathcal H}):\; e\subseteq \displaystyle\bigcup_{i=1}^q U_i\mbox{ and } |e\cap U_i|\leq 1\mbox{ for each } i\in[q]\right\}.$$
Note that ${\mathcal H}[U_1,\ldots, U_q]$ is an $r$-uniform $q$-partite hypergraph with parts $U_1,\ldots,U_q$.
By the symbol ${[n]\choose r}$, we mean the family of all $r$-subsets of the set $[n]$. 
The hypergraph $K^r_n=\left([n],{[n]\choose r}\right)$ is celled  the complete $r$-uniform hypergraph with $n$ vertices. For $r=2$, we would rather use $K_n$ instead of $K_n^2$.
The largest possible integer $n$ such that ${\mathcal H}$ 
contains $K^r_n$ as a  subhypergraph 
is called the {\it clique number of ${\mathcal H}$}, denoted $\omega({\mathcal H})$.

A {\it proper $t$-coloring} of a hypergraph ${\mathcal H}$ is a map $c:V({\mathcal H})\longrightarrow [t]$ such that there is no monochromatic edge. The minimum possible 
such a $t$ is called {\it the chromatic number of ${\mathcal H}$}, denoted $\chi({\mathcal H})$. If there is no such a $t$, we define the chromatic number to be infinite.
Let $c$ be a proper coloring  of ${\mathcal H}$ and  $U_1,\ldots, U_q$ be 
$q$ pairwise disjoint subsets  of $V({\mathcal H})$.
The hypergraph ${\mathcal H}[U_1,\ldots, U_q]$ is said to be {\it colorful } if  
for each $j\in[q]$, the vertices of $U_j$ get pairwise distinct colors.
For a properly colored graph $G$, a subgraph is called {\it multicolored} if its vertices get pairwise distinct colors.

For a hypergraph ${\mathcal H}$, {\it the Kneser hypergraph} ${\rm KG}^r({\mathcal H})$ is an $r$-uniform hypergraph with the vertex set $E({\mathcal H})$ and whose edges are formed by $r$ pairwise vertex-disjoint edges of ${\mathcal H}$, i.e.,
$$E({\rm KG}^r({\mathcal H}))=\left\{ \{e_1,\ldots,e_r\}:\; e_i\cap e_j=\varnothing\mbox{ for each } i\neq j\in[r] \right\}.$$
For any graph $G$,
it is known that there are several hypergraphs ${\mathcal H}$ such that ${\rm KG}^2({\mathcal H})$ and $G$ are isomorphic. 
The Kneser hypergraph ${\rm KG}^r\left(K_n^k\right)$ is called the ``usual" 
Kneser hypergraph which is denoted by ${\rm KG}^r(n,k)$. 
Coloring properties of Kneser hypergraphs have been studied extensively 
in the literature. Lov\'asz~\cite{MR514625} (for $r=2$) and Alon, Frankl and Lov\'asz~\cite{MR857448} determined the chromatic number of ${\rm KG}^r(n,k)$. For an integer $r\geq 2$, they proved that
$$\chi\left({\rm KG}^r(n,k)\right)= \left\lceil{n-r(k-1)\over r-1}\right\rceil.$$
For a hypergraph ${\mathcal H}$, the $r$-colorability defect of ${\mathcal H}$, denoted ${\rm cd}_r({\mathcal H})$, is the 
minimum number of vertices which should be removed such that the induced hypergraph on the remaining vertices is $r$-colorable, i.e., 
$${\rm cd}_r({\mathcal H})=\min\left\{|U|:\; {\mathcal H}[V({\mathcal H})\setminus U]\mbox{ is $r$-colorable}\right\}.$$
For a hypergraph ${\mathcal H}$, 
Dol'nikov~\cite{MR953021}~(for $r=2$) and K{\v{r}}{\'{\i}}{\v{z}}~{\rm \cite{MR1081939} proved taht 
$$\chi({\rm KG}^r({\mathcal H}))\geq \left\lceil{{\rm cd}_r({\mathcal H})\over r-1}\right\rceil,$$
which is a generalization of the results by Lov\'asz~\cite{MR514625} and Alon, Frankl and Lov\'asz~\cite{MR857448}.

For a positive integer $r$, let $\mathbb{Z}_r=\{\omega,\omega^2\ldots,\omega^r\}$ 
be a cyclic group of order $r$ with generator $\omega$.
Consider a vector  $X=(x_1,x_2,\ldots,x_n)\in(\mathbb{Z}_r\cup\{0\})^n$.
An alternating 
subsequence of $X$ is a sequence $x_{i_1},x_{i_2},\ldots,x_{i_m}$ 
of nonzero terms of $X$ such that $i_1<\cdots<i_m$ and $x_{i_j}\neq x_{i_{j+1}}$ for each $j\in [m-1]$. 
We denote by ${\rm alt}(x)$ the maximum possible length of an 
alternating subsequence of $X$.
For a vector $X=(x_1,x_2,\ldots,x_n)\in (\mathbb{Z}_r\cup\{0\})^n$ and for an $\epsilon\in\Z_p$, set 
$X^\epsilon=\{i\in[n]:\; x_i=\epsilon\}$.
Note that, by abuse of notation, we can write $X=(X^\epsilon)_{\epsilon\in \mathbb{Z}_r}$.
For two vectors $X,Y\in (\mathbb{Z}_r\cup\{0\})^n$, by $X\subseteq Y$, we mean
$X^\epsilon\subseteq Y^\epsilon$ for each $\epsilon\in\mathbb{Z}_r$.

For a hypergraph ${\mathcal H}$ and a bijection $\sigma:[n]\longrightarrow V({\mathcal H})$, define 
$${\rm alt}_r({\mathcal H},\sigma)=\ds\max\left\{{\rm alt}(X):\; X\in (\mathbb{Z}_r\cup\{0\})^n\mbox{ such that } E({\mathcal H}[\sigma(X^\epsilon)])=\varnothing\mbox{ for each  } \epsilon\in\mathbb{Z}_r \right\}.$$
Also, let 
$${\rm alt}_r({\mathcal H})=\displaystyle\min_{\sigma} {\rm alt}_r({\mathcal H},\sigma),$$
where the minimum is taken over all bijection $\sigma:[n]\longrightarrow V({\mathcal H})$.
One can readily check that for any hypergraph ${\mathcal H}$, $|V({\mathcal H})|-{\rm alt}_r({\mathcal H})\geq {\rm cd}_r({\mathcal H})$ and the inequality is often strict, see~\cite{2013arXiv1302.5394A}. 
Alishahi and Hajiabolhassan~\cite{2013arXiv1302.5394A} improved  Dol'nikov-K{\v{r}}{\'{\i}}{\v{z}} result by proving that 
for any hypergraph ${\mathcal H}$ and for any integer $r\geq 2$, the quantity $\left\lceil{ |V({\mathcal H})|-{\rm alt}_r({\mathcal H})\over r-1}\right\rceil$ is a lower bound for the chromatic number of 
${\rm KG}^r({\mathcal H})$.  
Using this lower bound, the chromatic number of some families of graphs and hypergraphs are computed, see~\cite{2014arXiv1401.0138A,2014arXiv1403.4404A,2014arXiv1407.8035A,2015arXiv150708456A,2013arXiv1302.5394A,HaMe16}. 
There are some other lower bounds for the chromatic number of graphs which are 
better than the former discussed lower bounds.
They are based on some topological  indices of some topological spaces 
connected to the structure of graphs.
In spite of these lower bounds being better, they are~not combinatorial and most of the times they are difficult to compute.

The existence of large colorful bipartite subgraphs in a properly colored graph has 
been extensively studied in the 
literature~\cite{2013arXiv1302.5394A,MR2971704,MR2763055,SiTaZs13,MR2279672,MR2351519}. 
To be more specific, there are several theorems ensuring the existence of a  colorful bipartite subgraph in any properly colored graph such that the bipartite subgraph has a specific number of vertices related to some topological parameters connected to the graph. 
Simonyi and Tardos~\cite{MR2351519} improved  Dol'nikov's lower bound and proved that 
in any proper coloring of a Kneser graph ${\rm KG}^2({\mathcal H})$, there is a multicolored complete bipartite graph 
$K_{\left\lceil{{\rm cd}_2({\mathcal H})\over 2}\right\rceil,\left\lfloor{{\rm cd}_2({\mathcal H})\over 2}\right\rfloor}$ 
such that the ${\rm cd}^2({\mathcal H})$ different colors occur alternating on the 
two parts of the bipartite graph with respect to their natural order. 
By a combinatorial proof, Alishahi and Hajiabolhassan~\cite{2013arXiv1302.5394A}  improved this result. They proved that the the result 
remains true if we replace ${\rm cd}^2({\mathcal H})$ by $n-{\rm alt}_2({\mathcal H})$. Also, 
a stronger result is proved by Simonyi, Tardif, and  Zsb{\'{a}}n~\cite{SiTaZs13}.
\begin{alphtheorem}{\rm (Zig-zag Theorem~\cite{SiTaZs13}).}\label{zigzag}
Let $G$ be a nonempty graph 
which is properly colored with arbitrary number of colors. 
Then $G$ contains a multicolored complete bipartite subgraph $K_{\lceil{t\over2}\rceil,\lfloor{t\over2}\rfloor}$,
where ${\rm Xind}({\rm Hom}(K_2,G))+2= t$. Moreover, colors appear alternating on the two sides of the bipartite subgraph with respect to their natural ordering.
\end{alphtheorem}

The quantity ${\rm Xind}({\rm Hom}(K_2,G))$ is the cross-index of 
hom-complex ${\rm Hom}(K^2,G)$  which will be defined in 
Subsection~\ref{Boxdefin}. We should mention that there are some other weaker similar 
results in terms of some other topological parameters, see~\cite{MR2279672,MR2351519}.
 
Note that prior mentioned results concern the existence of colorful bipartite subgraphs 
in properly colored graphs ($2$-uniform hypergraphs).
In 2014, Meunier~\cite{Meunier14} 
found the first colorful type result for the uniform hypergraphs.
He proved that for any prime number $p$, any properly colored Kneser 
hypergraph $\KG^p({\mathcal H})$ must contain a colorful
balanced complete $p$-uniform $p$-partite subhypergraph with a specific number 
of vertices, see Theorem~\ref{colorfulhyper}.

\subsection{\bf Main Results}
For a given graph $G$, 
there are several complexes defined based on the structure of $G$. 
For instance, the box-complex of $G$, denoted ${\rm B}_0(G)$,  and the hom-complex of $G$, denoted ${\rm Hom}(K_2,G)$, see~\cite{MR1988723,SiTaZs13,MR2279672}.
Also, there are some lower bounds for the chromatic number of graphs 
related to some indices of these complexes~\cite{SiTaZs13,MR2279672}.
In this paper, we naturally generalize the  definitions of  box-complex and hom-complex 
of graphs to uniform hypergraphs. 
Also, the definition of $\mathbb{Z}_p$-cross-index of $\Z_p$-posets will be introduced.
Using these complexes, 
as a first main result of this paper, we generalize Meunier's result~\cite{Meunier14} (Theorem~\ref{colorfulhyper})
to the following theorem.
\begin{theorem}\label{maincolorfulindex}
Let $r\geq 2$ be a positive integer and $p\geq r$ be a prime number.
Assume that ${\mathcal H}$ is an $r$-uniform hypergraph and $c:V({\mathcal H})\longrightarrow[C]$ is a proper coloring of 
${\mathcal H}$ {\rm (}$C$ arbitrary{\rm )}. Then we have the 
following assertions.
\begin{itemize}
\item[{\rm (i)}] There is some colorful balanced complete $r$-uniform $p$-partite 
subhypergraph in  ${\mathcal H}$
with ${\rm ind}_{\mathbb{Z}_p}(B_0({\mathcal H},\mathbb{Z}_p))+1$ vertices.
In particular,
$$\chi({\mathcal H})\geq {{\rm ind}_{\mathbb{Z}_p}(B_0({\mathcal H},\mathbb{Z}_p))+1\over r-1}.$$
\item[{\rm (ii)}] If $p\leq \omega({\mathcal H})$, then there is some colorful balanced complete $r$-uniform $p$-partite subhypergraph in ${\mathcal H}$ with ${\rm Xind}_{\mathbb{Z}_p}({\rm Hom}(K^r_p,{\mathcal H}))+p$ vertices.
In particular,
$$\chi({\mathcal H})\geq {{\rm Xind}_{\mathbb{Z}_p}({\rm Hom}(K^r_p,{\mathcal H}))+p\over r-1}.$$
\end{itemize}
\end{theorem}

Quantities ${\rm ind}_{\mathbb{Z}_p}(B_0({\mathcal H},\mathbb{Z}_p))$ and ${\rm Xind}_{\mathbb{Z}_p}({\rm Hom}(K^r_p,{\mathcal H}))$ appearing in the statement of Theorem~\ref{maincolorfulindex} are respectively
the $\mathbb{Z}_p$-index and  $\mathbb{Z}_p$-cross-index 
of the $\mathbb{Z}_p$-box-complex $B_0({\mathcal H},\mathbb{Z}_p)$  
and $\mathbb{Z}_p$-hom-complex ${\rm Hom}(K^r_p,{\mathcal H})$  which will be defined in Subsection~\ref{Boxdefin}. 
Using these complexes,  we introduce some new lower bounds  for the chromatic number of uniform hypergraphs.
In view of Theorem~\ref{maincolorfulindex},   
next theorem provides a hierarchy of lower bounds for the chromatic number of $r$-uniform hypergraphs.
\begin{theorem}\label{inequalities}
Let $r\geq 2$ be a positive integer and $p\geq r$ be a prime number.
For any $r$-uniform hypergraph ${\mathcal H}$, we have the following inequalities.\\  
{\rm (i)} If $p\leq \omega({\mathcal H})$, then
$${\rm Xind}_{\mathbb{Z}_p}({\rm Hom}(K^r_p,{\mathcal H}))+p \geq  {\rm ind}_{\mathbb{Z}_p}(B_0({\mathcal H},\mathbb{Z}_p))+1.$$
{\rm (ii)} If ${\mathcal H}={\rm KG}^r({\mathcal F})$ for some hypergraph ${\mathcal F}$, then 
$$
 {\rm ind}_{\mathbb{Z}_p}(B_0({\mathcal H},\mathbb{Z}_p))+1
\geq   |V({\mathcal F})|-{\rm alt}_p({\mathcal F})
\geq  {\rm cd}_p({\mathcal F}).
$$
\end{theorem}

In view of  Theorem~\ref{inequalities}, Theorem~\ref{maincolorfulindex} is a common 
extension of Theorem~\ref{zigzag} and Theorem~\ref{colorfulhyper}. Furthermore, for $r=2$, Theorem~\ref{maincolorfulindex} implies the next corollary which also is a generalization of Theorem~\ref{zigzag}.
\begin{corollary}\label{cor1}
Let $p$ be a prime number and let $G$ be a nonempty graph which is properly colored with arbitrary number of colors. Then  there is a multicolored 
complete $p$-partite subgraph $K_{n_1,n_2,\ldots,n_p}$ of $G$ such that 
\begin{itemize}
\item $\ds\sum_{i=1}^pn_i={\rm ind}_{\mathbb{Z}_p}(B_0(G,\mathbb{Z}_p))+1$,
\item $|n_i-n_j|\leq 1$ for each $i,j\in[p]$. 
\end{itemize}
Moreover, if $p\leq \omega(G)$, then ${\rm ind}_{\mathbb{Z}_p}(B_0(G,\mathbb{Z}_p))+1$ 
can be replaced with ${\rm Xind}_{\mathbb{Z}_p}({\rm Hom}(K_p,G))+p$.
\end{corollary}

In view of the prior mentioned results, the following question naturally arises.
\begin{question} Do Theorem~\ref{maincolorfulindex} and Theorem~\ref{inequalities} remain true for non-prime $p$?
\end{question}

\subsection{\bf Applications to Local Chromatic Number of Uniform Hypergraphs}
For a graph $G$ and a vertex $v\in V(G)$, the {\it closed neighborhood of $v$,} denoted $N[v]$, is the set $\{v\}\cup\{u:\; uv\in E(G)\}$. The {\it local chromatic number} of $G$, denoted 
$\chi_l(G)$, is defined in~\cite{ERDOS198621} as follows:
$$\chi_l(G)=\ds\min_c\max\{|c(N[v])|:\; v\in V(G)\}$$
where the minimum in taken over all proper coloring $c$ of $G$.
Note that Theorem~\ref{zigzag} gives the following lower bound for the local chromatic 
number of a nonempty graph $G$:
\begin{equation}\label{localloerzigzag}
\chi_l(G)\geq \left\lceil{{\rm Xind}({\rm Hom}(K_2,G))+2\over 2}\right\rceil+1.
\end{equation}
Note that for a Kneser hypergraph $\KG^2({\mathcal H})$,
by using Simonyi and Tardos colorful result~\cite{MR2351519} or the extension given by 
Alishahi and Hajiabolhassan~\cite{2013arXiv1302.5394A}, there are two similar lower 
bounds for $\chi_l(\KG^2({\mathcal H}))$ which respectively used $\cd_2({\mathcal H})$ and 
$|V({\mathcal H})|-\alt_2({\mathcal H})$ instead of ${\rm Xind}({\rm Hom}(K_2,G))+2$. 
However, as it is stated in Theorem~\ref{inequalities}, the lower bound in terms 
of ${\rm Xind}({\rm Hom}(K_2,G))+2$ is better than these two last mentioned lower bounds.
Using Corollary~\ref{cor1}, we have the following lower bound for the local 
chromatic number of graphs. 
\begin{corollary}\label{locallowerp}
Let $G$ be a nonempty graph and  $p$ be a prime number. 
Then $$\chi_l(G)\geq t-\left\lfloor{t\over p}\right\rfloor+1,$$ where 
$t={\rm ind}_{\mathbb{Z}_p}(B_0(G,\mathbb{Z}_p))+1$.
Moreover, if $p\leq \omega(G)$, then ${\rm ind}_{\mathbb{Z}_p}(B_0(G,\mathbb{Z}_p))+1$ 
can be replaced with ${\rm Xind}_{\mathbb{Z}_p}({\rm Hom}(K_p,G))+p$.
\end{corollary}
Note that if we set $p=2$, then previous theorem implies Simonyi 
and Tardos lower bound for the local chromatic number. Note that, in general, 
this lower bound might be better than Simonyi 
and Tardos lower bound. To see this, let $k\geq 2$ be a fixed integer. 
Consider the Kneser graph $\KG^2(n,k)$ and let $p=p(n)$ be a prime number such that
$p=O(\ln n)$.  By Theorem~\ref{inequalities},  for $n\geq pk$, we have  
$${\rm ind}_{\mathbb{Z}_p}(B_0(\KG^2(n,k),\mathbb{Z}_p))+1\geq \cd_p(K_n^k)=n-p(k-1).$$
Note that
the lower bound for $\chi_l(\KG^2(n,k))$ coming form Inequality~\ref{localloerzigzag}
is \begin{equation}\label{equation2}
1+\left\lceil{n-2k+2\over 2}\right\rceil={n\over 2}-o(1),
\end{equation} 
while, 
in view of Corollary~\ref{locallowerp}, we have
$$\chi_l(\KG^2(n,k))\geq n-p(k-1)-\left\lfloor{n-p(k-1)\over p}\right\rfloor+1=n-o(n),$$
which is better than the quantity in Equation~\ref{equation2} if $n$ is sufficiently large. 
However,   
since the induced subgraph on the neighbors of any vertex of $\KG(n,k)$ is isomorphic to
$\KG(n-k,k)$, we have
$$\chi_l(\KG(n,k))\geq n-3(k-1).$$
\begin{corollary}
Let $\F$ be a hypergraph and $\alpha(\F)$ be its independence number. Then for any prime number $p$, we have
$$\chi_l(\KG^2(\F))\geq \left\lceil{(p-1)|V(\F)|\over p}\right\rceil-(p-1)\cdot\alpha(\F)+1.$$
\end{corollary}
\begin{proof}
In view of Theorem~\ref{inequalities}, we have  
$${\rm ind}_{\mathbb{Z}_p}(B_0(\KG^2(\F),\mathbb{Z}_p))+1\geq \cd_p(\F)\geq |V(\F)|-p\cdot\alpha(\F).$$
Now, Corollary~\ref{locallowerp} implies the assertion.
\end{proof}
Meunier~\cite{Meunier14} naturally generalized the definition of 
local chromatic number of graphs 
to  uniform hypergraphs as follows. Let ${\mathcal H}$ be a uniform hypergraph. For a 
set $X\subseteq V({\mathcal H})$, the closed neighborhood of $X$ is the set $X\cup {\mathcal N}(X),$
where 
$${\mathcal N}(X)=\{v\in V({\mathcal H}):\; \exists\; e\in E({\mathcal H})\mbox{ such that } e\setminus X=\{v\}\}.$$
For a uniform hypergraph ${\mathcal H}$, the local chromatic number of ${\mathcal H}$ is defined as follows: 
$$\chi_l({\mathcal H})=\ds\min_c\max\{|c({\mathcal N}[e\setminus \{v\}])|:\; e\in E({\mathcal H})\mbox{ and } v\in e\},$$
where the minimum is taken over all proper coloring $c$ of ${\mathcal H}$.

Meunier~\cite{Meunier14}, by using his colorful theorem (Theorem~\ref{colorfulhyper}), 
generalized Simonyi and Tardos lower bound~\cite{MR2351519}  for the local chromatic 
number of Kneser graphs to the local chromatic number of Kneser hypergraphs. He proved: 
$$
\chi_l(\KG^p({\mathcal H}))\geq \min\left(\left\lceil{|V({\mathcal H})|-\alt_p({\mathcal H})\over p}\right\rceil+1, \left\lceil{|V({\mathcal H})|-\alt_p({\mathcal H})\over p-1}\right\rceil  \right)$$  for any hypergraph ${\mathcal H}$ and any prime number $p$.
In what follows, we generalize this result.
\begin{theorem}\label{localhyper}
Let ${\mathcal H}$ be an $r$-uniform hypergraph with at least one edge and $p$ be a prime number, 
where $r\leq p\leq \omega({\mathcal H})$. Let $t={\rm Xind}_{\mathbb{Z}_p}({\rm Hom}(K^r_p,{\mathcal H}))+p$.
If  $t=ap+b$, where $a$ and $b$ are nonnegative integers and
$0\leq b\leq p-1$, then
$$\chi_l({\mathcal H})\geq  \min\left(\left\lceil{(p-r+1)a+\min\{p-r+1,b\}\over r-1}\right\rceil+1, \left\lceil{t\over r-1}\right\rceil  \right).$$ 
\end{theorem}
\begin{proof}
Let $c$ be an arbitrary proper coloring of ${\mathcal H}$ and
let ${\mathcal H}[U_1,\ldots,U_p]$ be the colorful balanced complete $r$-uniform $p$-partite 
subhypergraph of ${\mathcal H}$ whose existence is ensured by Theorem~\ref{maincolorfulindex}. 
Note that $b$ numbers of $U_i$'s, say $U_1,\ldots,U_b$, have the cardinality 
$\lceil{t\over r}\rceil$ while the others 
have the cardinality $\lfloor{t\over r}\rfloor\geq 1$.
Consider $U_1,\ldots,U_{p-r+1}$. 
Two different cases will be distinguished. 
\begin{itemize}
\item[{\bf Case 1.}] If $\left|\ds\bigcup_{i=1}^{p-r+1} c(U_i)\right|<\left\lceil{t\over r-1}\right\rceil$,
then there is a vertex $v\in \ds\bigcup_{i=p-r+2}^p U_i$ whose color is~not in $\ds\bigcup_{i=1}^{p-r+1} c(U_i)$. Consider an edge of ${\mathcal H}[U_1,\ldots,U_p]$ containing $v$
and such that $|e\cap U_{p-r+1}|=1$ and $e\cap U_i=\varnothing$ for $i=1,\ldots, p-r$. 
Let $e\cap U_{p-r+1}=\{u\}$.
One can check that $$\{c(v)\}\cup\ds\left(\bigcup_{i=1}^{p-r+1} c(U_i)\right)\subseteq c({\mathcal N}(e\setminus \{u\})).$$ Therefore, 
since any color is appeared in at most $r-1$ number of $U_i$'s, we have
$$\left|\bigcup_{i=1}^{p-r+1} c(U_i)\right|\geq \ds\left\lceil{\sum_{i=1}^{p-r+1}|U_i|\over r-1}\right\rceil,$$
and consequently,
$$|c({\mathcal N}(e\setminus \{u\}))|\geq 1+\ds\left\lceil{\sum_{i=1}^{p-r+1}|U_i|\over r-1}\right\rceil=1+\left\lceil{(p-r+1)a+\min\{p-r+1,b\}\over r-1}\right\rceil,$$
which completes the proof in Case~1.
\item[{\bf Case 2.}] If $\left|\ds\bigcup_{i=1}^{p-r+1} c(U_i)\right|\geq\left\lceil{t\over r-1}\right\rceil$,
then consider an edge of ${\mathcal H}[U_1,\ldots,U_p]$ such that 
$|e\cap U_{p-r+1}|=1$ and $e\cap U_i=\varnothing$ for $i=1,\ldots, p-r$. 
Let $e\cap U_{p-r+1}=\{u\}$.
One can see that $$\ds\bigcup_{i=1}^{p-r+1} c(U_i)\subseteq c({\mathcal N}(e\setminus \{u\})),$$
which completes the proof in Case~2.
\end{itemize}
\end{proof}

\begin{corollary}
Let ${\mathcal H}$ be a $p$-uniform hypergraph with at least one edge, where $p$ is a prime number.
Then
$$\chi_l({\mathcal H})\geq \min\left(\left\lceil{{\rm Xind}_{\mathbb{Z}_p}({\rm Hom}(K^p_p,{\mathcal H}))+p\over p}\right\rceil+1, \left\lceil{{\rm Xind}_{\mathbb{Z}_p}({\rm Hom}(K^p_p,{\mathcal H}))+p\over p-1}\right\rceil  \right).$$ 
\end{corollary}
\begin{proof}
Since ${\mathcal H}$ has at least one edge, we have $\omega({\mathcal H})\geq p$. Therefore, in view of Theorem~\ref{localhyper}, we have the assertion.
\end{proof}
Note that if ${\mathcal H}=\KG^p(\F)$, then, in view of Theorem~\ref{inequalities},
we have 
$${\rm Xind}_{\mathbb{Z}_p}({\rm Hom}(K^p_p,{\mathcal H}))+p\geq |V(\F)|-\alt_p(\F).$$
This implies that the previous corollary is a generalization of Meunier's lower bound for the
local chromatic number of $\KG^p(\F)$

\subsection{\bf Plan}
Section~\ref{intro} contains some backgrounds and essential definitions used 
elsewhere in the paper.
In Section~\ref{definition}, we present some new topological 
tools which help us for the proofs of main results. 
Section~\ref{sec:proofs} is devoted to the proofs of 
Theorem~\ref{maincolorfulindex} and Theorem~\ref{inequalities}.

\section{\bf Preliminaries}\label{intro}

\subsection{{\bf Topological Indices and Lower Bound for Chromatic Number}}

We assume basic knowledge in combinatorial algebraic topology. 
Here, we are going to bring a brief review of some essential notations and definitions which will be needed throughout the paper. For more, one can see the book written by Matou{\v{s}}ek~\cite{MR1988723}. Also, the definitions of box-complex, hom-complex, and cross-index
will be generalized to $\Z_p$-box-complex, $\Z_p$-hom-complex, and $\Z_p$-cross-index, respectively.

Let $\mathbb{G}$ be a finite nontrivial group which acts on a topological space $X$.
We call $X$ a {\it topological $\mathbb{G}$-space} if for each $g\in \mathbb{G}$, the map  $g:X\longrightarrow X$ which $x\mapsto g\cdot x$ is continuous. A {\it free topological $\mathbb{G}$-space $X$} is a topological $\mathbb{G}$-space such that $\mathbb{G}$ acts on it freely, i.e., for each $g\in \mathbb{G}\setminus\{e\}$, the map $g:X\longrightarrow X$ has no fixed point. For two topological $\mathbb{G}$-spaces $X$ and $Y$, a continuous map
$f:X\longrightarrow Y$ is called a $\mathbb{G}$-map if $f(g\cdot x)=g\cdot f(x)$ for each $g\in \mathbb{G}$ and $x\in X$. We write $X\stackrel{\mathbb{G}}{\longrightarrow} Y$ to mention that there is a $\mathbb{G}$-map from $X$ to $Y$. A map $f:X\longrightarrow Y$ is called a $\mathbb{G}$-equivariant map, if $f(g\cdot x)=g\cdot f(x)$ for each $g\in \mathbb{G}$ and $x\in X$.

Simplicial complexes provide
a bridge between combinatorics and topology. A simplicial complex can be viewed
as a combinatorial object, called abstract simplicial complex, or as a topological
space, called geometric simplicial complex. Here, we just remind the definition of an
abstract simplicial complex. However, we assume that the reader is familiar with the concept of how an abstract simplicial complex and its geometric realization are connected to each other. 
A {\it simplicial complex} is a pair $(V,K)$, where $V$ is a finite set and $K$ is a family of subsets of $V$ such that if $F\in K$ and $F'\subseteq F$, then $F'\in K$. 
Any set in $K$ is called a simplex. 
Since we may assume that $V=\bigcup_{F\in K}F$,
we can write $K$ instead of $(V, K)$.
The {\it dimension of $K$} is defined as follows:
$${\rm dim}(K)=\max\{|F|-1:\; F\in K\}.$$ 
The geometric realization of $K$ is denoted by $||K||$.
For two simplicial complexes $C$ and $K$, by a {\it simplicial map $f:C\longrightarrow K$,} 
we mean a map from $V(C)$ to $V(K)$ such that the image of any 
simplex of $C$ is a simplex of $K$. 
For a nontrivial finite group $\mathbb{G}$, a {\it simplicial $\mathbb{G}$-complex} $K$ 
is a simplicial complex with a $\mathbb{G}$-action on its vertices such that 
each $g\in \mathbb{G}$ induces a simplicial map from $K$ to $K$, that is
the map which maps $v$ to $g\cdot v$ for each $v\in V(K)$.
If for each $g\in \mathbb{G}\setminus\{e\}$, there is no fixed simplex under the 
simplicial map made by $g$, then $K$ is called a 
{\it free simplicial $\mathbb{G}$-complex.} For a simplicial $\mathbb{G}$-complex $K$,
if we take the affine extension, then $K$ is free if and only if $||K||$ is free.
For two simplicial $\mathbb{G}$-complexes $C$ and $K$, a simplicial map
$f:C\longrightarrow K$ is called a simplicial $\mathbb{G}$-map if $f(g\cdot v)=g\cdot f(v)$ 
for each $g\in \mathbb{G}$ and $v\in V(C)$. We write 
$C\stackrel{\mathbb{G}} {\longrightarrow} K$, if there is a simplicial $\mathbb{G}$-map 
from $C$ to $K$. Note that if $C\stackrel{\mathbb{G}}{\longrightarrow} K$, 
then $||C||\stackrel{\mathbb{G}}{\longrightarrow} ||K||$. A map $f:C\longrightarrow K$ 
is called a $\mathbb{G}$-equivariant map, if $f(g\cdot v)=g\cdot f(v)$ 
for each $g\in \mathbb{G}$ and $v\in V(C)$.

For an integer $n\geq 0$ and a nontrivial finite group $\mathbb{G}$, 
{\it $E_n \mathbb{G}$ space} is a  
free $(n-1)$-connected $n$-dimensional simplicial $\mathbb{G}$-complexes. 
A concrete example of an $E_n \mathbb{G}$ space is the $(n+1)$-fold 
join $\mathbb{G}^{*(n+1)}$. 
As a topological space $\mathbb{G}^{*(n+1)}$ is a $(n+1)$-fold join of 
an $(n+1)$-point discrete space.  
This is known that for any two $E_n \mathbb{G}$ space $X$ and $Y$,
there is a $\mathbb{G}$-map from $X$ to $Y$.

For a $\mathbb{G}$-space $X$, define 
$${\rm ind}_{\mathbb{G}}(X)=\min\{n:\; X\stackrel{\mathbb{G}}{\longrightarrow}
E_n\mathbb{G}\}.$$
Note that here $E_n \mathbb{G}$ can be any $E_n \mathbb{G}$, since there is a 
$\mathbb{G}$-map between any two $E_n \mathbb{G}$ spaces, see~\cite{MR1988723}. Also, for a simplicial 
complex $K$, by ${\rm ind}_{\mathbb{G}}(K)$, we mean 
${\rm ind}_{\mathbb{G}}(||K||)$. Throughout the paper,
for $\mathbb{G}=\mathbb{Z}_2$, we would rather use ${\rm ind}(-)$ 
instead of ${\rm ind}_{\mathbb{Z}_2}(-)$.

\noindent{\bf Properties of the $\mathbb{G}$-index.} \cite{MR1988723} Let $\mathbb{G}$ be a finite nontrivial group.
\begin{itemize}
\item[{\rm (i)}] ${\rm ind}_{\mathbb{G}}(X)>{\rm ind}_{\mathbb{G}}(Y)$ implies $X\stackrel{\mathbb{G}}{\centernot\longrightarrow} Y$.
\item[{\rm (ii)}] ${\rm ind}_{\mathbb{G}}(E_n \mathbb{G})=n$ for any $E_n \mathbb{G}$ space.
\item[{\rm (iii)}] ${\rm ind}_{\mathbb{G}}(X*Y)\leq {\rm ind}_{\mathbb{G}}(X)+{\rm ind}_{\mathbb{G}}(Y)+1$.
\item[{\rm (iv)}] If $X$ is $(n-1)$-connected, then ${\rm ind}_{\mathbb{G}}(X)\geq n$.
\item[{\rm (v)}] If $K$ is a free simplicial $\mathbb{G}$-complex of dimension $n$, then ${\rm ind}_{\mathbb{G}}(K)\leq n$.
\end{itemize}

\subsection{{\bf $\mathbb{Z}_p$-Box-Complex, $\mathbb{Z}_p$-Poset, and $\mathbb{Z}_p$-Hom-Complex}}\label{Boxdefin}
In this subsection, for any $r$-uniform hypergraph ${\mathcal H}$, we are going to define two objects; $\mathbb{Z}_p$-box-complex of ${\mathcal H}$ and $\mathbb{Z}_p$-hom-complex of ${\mathcal H}$  which the first one is a
simplicial $\mathbb{Z}_p$-complex and the second one is a $\mathbb{Z}_p$-poset.
Moreover, for any $\mathbb{Z}_p$-poset $P$, we assign a combinatorial index to $P$ called the cross-index of $P$.
\\

\noindent{\bf $\mathbb{Z}_p$-Box-Complex.}
Let $r\geq 2$ be a positive integer and $p\geq r$ be a prime number.
For an $r$-uniform hypergraph ${\mathcal H}$, define the {\it $\mathbb{Z}_p$-box-complex of ${\mathcal H}$,} denoted ${\rm B}_0({\mathcal H},{\mathbb{Z}_p})$,  
to be a simplicial complex with the vertex set $\displaystyle\biguplus_{i=1}^pV({\mathcal H})=\mathbb{Z}_p\times V({\mathcal H})$ and  the simplex set consisting of all
$\{\omega^1\}\times U_1\cup\cdots\cup \{\omega^p\}\times U_p$,
where 
 \begin{itemize}
\item $U_1,\ldots,U_p$ are pairwise disjoint subsets of $V({\mathcal H})$,
\item $\displaystyle\bigcup_{i=1}^p U_i\neq\varnothing$, and
\item the hypergraph ${\mathcal H}[U_1,U_2,\ldots,U_p]$
is a complete $r$-uniform $p$-partite hypergraph.
 \end{itemize}
Note that some of $U_i$'s might be empty. 
In fact, if $U_1,\ldots,U_p$ are pairwise disjoint subsets of $V({\mathcal H})$ and the number of nonempty $U_i$'s is less than $r$, then ${\mathcal H}[U_1,U_2,\ldots,U_p]$ is a complete $r$-uniform $p$-partite hypergraph and thus
$\{\omega^1\}\times U_1\cup\cdots\cup \{\omega^p\}\times U_p\in {\rm B}_0({\mathcal H},{\mathbb{Z}_p})$. For each $\epsilon\in{\mathbb{Z}_p}$ 
and each $(\epsilon',v)\in V({\rm B}_0({\mathcal H},{\mathbb{Z}_p}))$, define  $\epsilon\cdot(\epsilon',v)=(\epsilon\cdot\epsilon',v)$. One can see that this action makes   
${\rm B}_0({\mathcal H},{\mathbb{Z}_p})$ a free simplicial $\mathbb{Z}_p$-complex. 
It should be mentioned that the $\Z_2$-box-complex ${\rm B}_0({\mathcal H},\Z_2)$ is extensively 
studied in the literature, see~\cite{MR2279672,MR2351519}.  
In the literature, for a graph $G$, the simplicial complex 
${\rm B}_0(G,{\mathbb{Z}_2})$ is shown by $B_0(G)$.
This simplicial complex is used to introduce some lower bounds for the chromatic 
number of a  given graph $G$, see~\cite{MR2279672}. 
In particular, we have the following inequalities
$$\chi(G)\geq {\rm ind}(B_0(G))+1\geq {\rm coind}(B_0(G))+1\geq n-{\rm alt}({\mathcal F})\geq{\rm cd}_2({\mathcal F}),$$
where ${\mathcal F}$ is any hypergraph such that ${\rm KG}^2({\mathcal F})$ and $G$ are 
isomorphic, see~\cite{2014arXiv1403.4404A,2013arXiv1302.5394A,MR2279672}.\\

\noindent{\bf $\mathbb{Z}_p$-Poset.}
A partially ordered set, or simply a {\it poset}, is defined as an ordered pair $P=(V(P),\preceq)$, where  $V(P)$ is a set called the ground set of $P$ and $\preceq$ is a partial order on $V(P)$.
For two posets $P$ and $Q$, 
by an order-preserving map $\phi:P\longrightarrow Q$, we mean  a map $\phi$  
from  $V(P)$ to $V(Q)$  such that for each $u,v\in V(P)$, if $u\preceq v$, then $\phi(u)\preceq \phi(v)$.
 A poset $P$ is called a {\it  $\mathbb{Z}_p$-poset}, if $\mathbb{Z}_p$ acts on $V(P)$ 
 and furthermore,
for each $\epsilon\in \mathbb{Z}_p$, the map $\epsilon:V(P)\longrightarrow V(P)$ which $v\mapsto \epsilon\cdot v$ is an automorphism of $P$ (order preserving bijective map).
If for each $\epsilon\in \mathbb{Z}_p\setminus\{e\}$, this map has no fixed  point, then $P$ is called a {\it  free $\mathbb{Z}_p$-poset}.
For two $\mathbb{Z}_p$-poset $P$ and $Q$, 
by an order-preserving $\mathbb{Z}_p$-map $\phi:P\longrightarrow Q$, we mean 
 an order-preserving map from  $V(P)$ to $V(Q)$ such that for each $v\in V(P)$ and $\epsilon\in \mathbb{Z}_p$, we have $\phi(\epsilon\cdot v)=\epsilon\cdot\phi(v)$.
 If there exists such a map, we write $P\stackrel{\mathbb{Z}_p}{\longrightarrow} Q$.

For a nonnegative integer $n$ and a prime number $p$, let $Q_{n,p}$ be a free
$\mathbb{Z}_p$-poset with ground set $\mathbb{Z}_p\times[n+1]$ such that for any 
two members $(\epsilon,i),(\epsilon',j)\in Q_{n,p}$,  $(\epsilon,i)<_{Q_{n,p}}(\epsilon',j)$ if
$i<j$. Clearly, $Q_{n,p}$ is a   free $\mathbb{Z}_p$-poset with the action $\epsilon\cdot(\epsilon',j)=(\epsilon\cdot\epsilon',j)$ for each $\epsilon\in\mathbb{Z}_p$ and $(\epsilon',j)\in Q_{n,p}$.
For a $\mathbb{Z}_p$-poset $P$, the {\it $\mathbb{Z}_p$-cross-index} of $P$, denoted
${\rm Xind}_{\mathbb{Z}_p}(P)$, is the least integer $n$ such that there is a 
$\mathbb{Z}_p$-map from $P$ to $Q_{n,p}$. Throughout the paper, for $p=2$,  
we speak about ${\rm Xind}(-)$ rather than ${\rm Xind}_{\mathbb{Z}_2}(-)$. 
It should be mentioned that ${\rm Xind}(-)$ is first defined in~\cite{SiTaZs13}.

Let $P$ be a poset. We can define an {\it order complex} $\Delta P$ with the vertex set 
same as the ground set of $P$ and simplex set consisting of all chains in $P$. 
One can see that if $P$ is a free $\mathbb{Z}_p$-poset, then $\Delta P$ is a free 
simplicial $\mathbb{Z}_p$-complex. 
Moreover, any order-preserving $\Z_p$-map $\phi:P\longrightarrow Q$ can be lifted 
to a simplicial 
$\mathbb{Z}_p$-map from $\Delta P$ to $\Delta Q$. Clearly, there is a simplicial 
$\mathbb{Z}_p$-map from $\Delta Q_{n,p}$ to $\mathbb{Z}_p^{*(n+1)}$ (identity map). 
Therefore,
if ${\rm Xind}_{\mathbb{Z}_p}(P)=n$, then
we have a simplicial $\mathbb{Z}_p$-map from $\Delta P$ to $\mathbb{Z}_p^{*(n+1)}$. This 
implies that 
${\rm Xind}_{\mathbb{Z}_p}(P)\geq {\rm ind}_{\mathbb{Z}_p}(\Delta P)$. 
Throughout the paper, for each $(\epsilon, j)\in Q_{n,p}$, when we speak about the sign of  
$(\epsilon, j)$ and the absolute value of  $(\epsilon, j)$, we mean $\epsilon$ and 
$j$, respectively.

\begin{alphtheorem}{\rm \cite{AliHajiMeu2016}}\label{altercrossindex}
Let $P$ be a free ${\mathbb Z}_2$-poset and $\phi:P\longrightarrow Q_{s,2}$ be an order preserving ${\mathbb Z}_2$-map. Then $P$ contains a chain $p_1\prec_P\cdots\prec_Pp_{ k}$ such that $k= {\rm Xind}(P)+1$ and
the signs of $\phi(p_i)$ and $\phi(p_{i+1})$ differ 
for each $i\in[k-1]$. Moreover, if $s= {\rm Xind}(P)$, then for any $(s+1)$-tuple 
$(\epsilon_1,\ldots,\epsilon_{s+1})\in\mathbb{Z}_2^{s+1}$, there is at least one chain 
$p_1\prec_P\cdots\prec_Pp_{ s+1}$ such that $\phi(p_i)=(\epsilon_i,i)$ for each $i\in[s+1]$.
\end{alphtheorem}

\noindent{\bf $\mathbb{Z}_p$-Hom-Complex.}
Let ${\mathcal H}$ be an $r$-uniform hypergraph. Also, let $p\geq r$ be a prime number. 
The {\it $\mathbb{Z}_p$-hom-complex} ${\rm Hom}(K^r_p,{\mathcal H})$ is a free 
$\mathbb{Z}_p$-poset with the ground set consisting of all ordered $p$-tuples 
$(U_1,\cdots,U_p)$, where $U_i$'s are nonempty pairwise disjoint subsets of $V$ and 
${\mathcal H}[U_1,\ldots,U_p]$ is a complete $r$-uniform $p$-partite hypergraph. 
For two $p$-tuples $(U_1,\cdots,U_p)$ and $(U'_1,\cdots,U'_p)$ in 
${\rm Hom}(K^r_p,{\mathcal H})$, we define $(U_1,\cdots,U_p)\preceq(U'_1,\cdots,U'_p)$ 
if $U_i\subseteq U'_i$ for each $i\in[p]$. 
Also, for each $\omega^i\in \mathbb{Z}_p=\{\omega^1,\ldots,\omega^p\}$, let 
$\omega^i\cdot (U_1,\cdots,U_p)=(U_{1+i},\cdots,U_{p+i})$, where $U_j=U_{j-p}$ for $j>p$. 
Clearly, this action is a free $\mathbb{Z}_p$-action on ${\rm Hom}(K^r_p,{\mathcal H})$. 
Consequently, ${\rm Hom}(K^r_p,{\mathcal H})$ is a free $\mathbb{Z}_p$-poset with 
this $\mathbb{Z}_p$-action.

For a nonempty graph $G$ and for $p=2$, it is proved~\cite{2014arXiv1403.4404A,2013arXiv1302.5394A,SiTaZs13,MR2279672} that
\begin{equation}\label{equation}
\begin{array}{lll}
\chi(G) &\geq & {\rm Xind}({\rm Hom}(K_2,G))+2 \geq  {\rm ind}(\Delta {\rm Hom}(K_2,G))+2 \geq  {\rm ind}(B_0(G))+1\\
&\geq &  {\rm coind}(B_0(G))+1\geq  |V({\mathcal F})|-{\rm alt}_2({\mathcal F})
\geq  {\rm cd}_2({\mathcal F}),
\end{array}
\end{equation} 
where ${\mathcal F}$ is any hypergraph such that ${\rm KG}^2({\mathcal F})$ and $G$ are isomorphic.\\

\section{\bf Notations and Tools}\label{definition}

For a simplicial complex $K$, by $\sd K$, we mean the first barycentric subdivision of $K$.
It is the simplicial complex whose vertex set is the set of nonempty simplices of $K$
and whose simplices are the collections of simplices of $K$ which are pairwise 
comparable by inclusion.  Throughout the paper, by $\sigma_{t-1}^{r-1}$, we mean the 
$(t-1)$-dimensional simplicial complex with vertex set $\mathbb{Z}_r$  
containing all $t$-subsets 
of $\mathbb{Z}_r$ as its maximal simplices. 
The join of two simplicial complexes $C$ and $K$, denoted $C*K$, is a simplicial complex with the vertex set $V(C)\biguplus V(K)$ and such that the set of its simplices is
$\{F_1\biguplus F_2:\; F_1\in C\mbox{ and } F_2\in K\}$.
Clearly, we can see $\mathbb{Z}_r$ as a $0$-dimensional simplicial complex. 
Note that the vertex set of simplicial complex $\sd\mathbb{Z}_r^{*\alpha}$ can be identified with $(\mathbb{Z}_r\cup\{0\})^\alpha\setminus\{\zero\}$ and the vertex set of $(\sigma^{r-1}_{t-1})^{*n}$ is the set of all pairs $(\epsilon,i)$, where $\epsilon\in \mathbb{Z}_r$ and $i\in [n]$.

\subsection{{\bf $\mathbb{Z}_p$-Tucker-Ky Fan lemma}}
The famous Borsuk-Ulam theorem has many generalizations which have been extensively used in investigating graph coloring properties. Some of these interesting generalizations are Tucker lemma~\cite{MR0020254}, $Z_p$-Tucker Lemma~\cite{MR1893009}, and Tucker-Ky Fan~\cite{MR0051506}.  For more details about the Borsuk-Ulam theorem and its generalizations, we refer the reader to \cite{MR1988723}.

Actually, Tucker lemma is a combinatorial counterpart of Borsuk-Ulam theorem. There are several interesting and surprising applications of Tucker Lemma in combinatorics, including a combinatorial proof of Lov{\'a}sz-Kneser 
theorem by Matou{\v{s}}ek \cite{MR2057690}. 
\begin{alphlemma}\label{tuckeroriginal}
{\rm(Tucker lemma \cite{MR0020254}).} 
Let $m$ and $n$ be positive integers and
$\lambda:\{-1,0,+1\}^n\setminus \{(0,\ldots,0)\} \longrightarrow \{\pm 1, \pm 2,\ldots ,\pm m\}$
be a map satisfying the following properties:
\begin{itemize}
\item for any $X\in \{-1,0,+1\}^n\setminus \{\zero\}$, we have $\lambda(-X)=-\lambda(X)$ {\rm (}a
      $Z_2$-equivariant map{\rm ),}
\item no two signed vectors $X$ and $Y$ are such that
$X\subseteq Y$ and $\lambda(X) =-\lambda(Y)$.
\end{itemize}
Then, we have $m \geq n$.
\end{alphlemma}
Another interesting generalization of the Borsuk-Ulam theorem is
Ky~Fan's lemma~\cite{MR0051506}. This generalization ensures that with the same assumptions as in Lemma~\ref{tuckeroriginal}, there is odd number of chains
$X_1\subseteq X_2\subseteq \cdots \subseteq X_n$ such that $$\{\lambda(X_1),\ldots,\lambda(X_n)\}=\{+c_1,-c_2,\ldots ,
(- 1)^{n-1}c_n\},$$ where  $1\leq c_1 < \cdots < c_n \leq m$.
Ky~Fan's lemma has been used in several articles to
study some coloring properties of graphs, see \cite{AliHajiMeu2016,MR2763055,MR2837625}.
There are also some other generalizations of
Tucker Lemma. 
A $\mathbb{Z}_p$ version of Tucker Lemma, called $\mathbb{Z}_p$-Tucker Lemma, is proved by Ziegler~\cite{MR1893009} and extended by Meunier~\cite{MR2793613}.
In next subsection, we present a $\mathbb{Z}_p$ version of Ky~Fan's lemma
which is called $\mathbb{Z}_p$-Tucker-Ky Fan lemma.

\subsection{{\bf New Generalizations of Tucker Lemma}}
Before presenting our results, we need to introduce 
some functions having key roles in the paper.
Throughout the paper, we are going to use these functions repeatedly. 
Let $m$ be a positive integer.
We remind that $(\sigma^{p-1}_{p-2})^{*m}$ is a free simplicial $\mathbb{Z}_p$-complex with vertex set $\mathbb{Z}_p\times [m]$. \\

\noindent{\bf The value function $l(-)$.}
Let $\tau\in (\sigma^{p-1}_{p-2})^{*m}$ be a simplex.
For each $\epsilon\in \mathbb{Z}_p$, define
$\tau^\epsilon=\left\{(\epsilon,j):\; (\epsilon,j)\in \tau\right\}.$ 
Moreover,  define 
$$l(\tau)=\max\left\{\displaystyle|\bigcup_{\epsilon\in\mathbb{Z}_p} B^\epsilon|:\; \forall\epsilon\in\mathbb{Z}_p ,\; B^\epsilon\subseteq \tau^\epsilon\mbox{ and } \forall \epsilon_1,\epsilon_2\in\mathbb{Z}_p,\; |\;|B^{\epsilon_1}|-|B^{\epsilon_2}|\;|\leq 1 \right\}.$$
Note that if we set $h(\tau)=\ds\min_{\epsilon\in \mathbb{Z}_p}|\tau^\epsilon|$, then
$$l(\tau)=p\cdot h(\tau)+|\{\epsilon\in\mathbb{Z}_p:\; |\tau^\epsilon|>h(\tau)\}|.$$\\

\noindent{\bf The sign functions $s(-)$ and $s_0(-)$.} 
For an $a\in[m]$, 
let $W_a$ be the set of all simplices $\tau\in (\sigma_{p-2}^{p-1})^{*m}$ such that $|\tau^\epsilon|\in\{0,a\}$ for each $\epsilon\in\mathbb{Z}_p$. Let $W=\displaystyle\bigcup_{a=1}^{m}W_a$. 
Choose an arbitrary $\mathbb{Z}_p$-equivariant map $s:W\longrightarrow \mathbb{Z}_p$. 
Also, consider an $\mathbb{Z}_p$-equivariant map $s_0:\sigma_{p-2}^{p-1}\longrightarrow \mathbb{Z}_p$.
Note that since $\mathbb{Z}_p$ acts freely on both $\sigma_{p-2}^{p-1}$ and $W$, 
these maps can be easily built by choosing one representative in each orbit. It should be mentioned that both functions $s(-)$ and $s_0(-)$ are first introduced in~\cite{Meunier14}.

Now, we are in a position to generalize Tucker-Ky Fan lemma to $\mathbb{Z}_p$-Tucker-Ky Fan lemma.
\begin{lemma}{\rm ($\mathbb{Z}_p$-Tucker-Ky Fan lemma).}\label{Z_pfanlemma}
Let $m,n,p$ and $\alpha$ be nonnegative integers, where
$m,n\geq 1$, $m\geq \alpha\geq 1$, and $p$ is prime. Let
$$
\begin{array}{crcl}
\lambda: & (\mathbb{Z}_p\cup\{0\})^n\setminus\{\zero\} &\longrightarrow & \mathbb{Z}_p\times[m]\\
& X &\longmapsto & (\lambda_1(X),\lambda_2(X))
\end{array}$$ be a $\mathbb{Z}_p$-equivariant map
satisfying the following conditions. 
\begin{itemize}
\item For $X_1\subseteq X_2\in \left(\mathbb{Z}_p\cup\{0\}\right)^n\setminus\{\zero\}$,
if $\lambda_2(X_1)=\lambda_2(X_2)\leq \alpha$, then $\lambda_1(X_1)=\lambda_1(X_2)$.
\item For $X_1\subseteq X_2\subseteq\cdots \subseteq X_p\in \left(\mathbb{Z}_p\cup\{0\}\right)^n\setminus\{\zero\}$,
if $\lambda_2(X_1)=\lambda_2(X_2)=\cdots=\lambda_2(X_p)\geq\alpha+1$, then 
$$\left|\left\{\lambda_1(X_1),\lambda_1(X_2),\ldots,\lambda_1(X_p)\right\}\right|<p.$$
\end{itemize}
Then there is a chain $$Z_1\subset Z_2\subset\cdots\subset Z_{n-\alpha}\in \left(\mathbb{Z}_p\cup\{0\}\right)^n\setminus\{\zero\}$$
such that 
\begin{enumerate}
\item for each $i\in [n-\alpha]$, $\lambda_2(Z_i)\geq \alpha+1$,
\item  for each $i\neq j\in [n-\alpha]$, $\lambda(Z_i)\neq \lambda(Z_j)$, and
\item\label{condition3} for  each 
$\epsilon\in\mathbb{Z}_p$, 
$$\left\lfloor{n-\alpha\over p}\right\rfloor\leq \left|\left\{j:\; \lambda_1(Z_j)=\epsilon\right\}\right|\leq \left\lceil{n-\alpha\over p}\right\rceil.$$
\end{enumerate}
In particular, $n-\alpha\leq (p-1)(m-\alpha)$.
\end{lemma}
\begin{proof}
Note that the map $\lambda$ can be considered as a simplicial $\mathbb{Z}_p$-map from $\sd \mathbb{Z}_p^{*n}$ to 
$(\mathbb{Z}_p^{*\alpha})*((\sigma_{p-2}^{p-1})^{*(m-\alpha)}).$
Let $K={\rm Im}(\lambda)$. 
Note that each simplex in $K$ can be represented in a unique form $\sigma\cup\tau$ such that
$\sigma\in \mathbb{Z}_p^{*\alpha}$ and $\tau \in (\sigma_{p-2}^{p-1})^{*m-\alpha}.$

In view of definition of the function $l(-)$ and the properties which $\lambda$ satisfies in, to prove the assertion, it suffices to show that there is a simplex $\sigma\cup\tau\in K$ such that  $l(\tau)\geq n-\alpha$. For a contradiction, suppose that for each $\sigma\cup\tau\in K$, we have $l(\tau)\leq n-\alpha-1$.

Define the map 
$$\Gamma: \sd K\longrightarrow \mathbb{Z}_p^{*(n-1)}$$
such that for each vertex $\sigma\cup\tau\in V(\sd K)$, 
\begin{itemize}
\item if $\tau=\varnothing$, then $\Gamma(\sigma\cup\tau)=(\epsilon, j)$, 
where $j$ is the maximum possible value such that $(\epsilon, j)\in\sigma$.
Note that since $\sigma\in \mathbb{Z}_p^{*\alpha}$,  
there is only one $\epsilon\in\mathbb{Z}_p$ for which 
 the maximum is attained. Therefore, in this case, the function $\Gamma$ is well-defined.

\item if $\tau\neq\varnothing$. Define $h(\tau)=\ds\min_{\epsilon\in \mathbb{Z}_p}|\tau^\epsilon|.$
\begin{enumerate}[label={\rm (\roman*)}]
\item   If $h(\tau)=0$, then define $\bar{\tau}=\{\epsilon\in \mathbb{Z}_p:\; 
                      \tau^\epsilon=   \varnothing\}\in \sigma^{p-1}_{p-2}$ and 
                      $$\Gamma(\sigma\cup\tau)=\left(s_0(\bar\tau), \alpha+l(\tau)\right).$$
              
\item  If $h(\tau)> 0$, then define $\bar{\tau}=\displaystyle
	        \bigcup_{\{\epsilon\in\mathbb{Z}_p:\; |\tau^\epsilon|=h(\tau)\}} \tau^\epsilon\in W$ 
	        and $$\Gamma(\sigma\cup\tau)=\left(s(\bar\tau), \alpha+l(\tau)\right).$$

\end{enumerate}
\end{itemize}
Now, we claim that  $\Gamma$ is a simplicial $\mathbb{Z}_p$-map from $\sd K$ to $\mathbb{Z}_p^{*(n-1)}$. It is clear that  $\Gamma$ is a $\mathbb{Z}_p$-equivariant map.
For a contradiction, suppose that there are $\sigma\cup\tau,\sigma'\cup\tau' \in \sd K$
such that $\sigma\subseteq \sigma'$, $\tau\subseteq\tau'$, $\Gamma(\sigma\cup\tau)=(\epsilon,\beta)$, and $\Gamma(\sigma'\cup\tau')=(\epsilon',\beta)$, where $\epsilon\neq \epsilon'$. 
First note that in view of the definition of $\Gamma$ and the assumption $\Gamma(\sigma\cup\tau)=(\epsilon,\beta)$ and $\Gamma(\sigma'\cup\tau')=(\epsilon',\beta)$, the case  $\tau=\varnothing$ and $\tau'\neq\varnothing$ is not possible.
If $\tau'=\varnothing$, then 
$\tau=\tau'=\varnothing$ and we should have 
$(\epsilon,\beta),(\epsilon',\beta)\in\sigma'\in \mathbb{Z}_p^{*\alpha}$ which implies that $\epsilon=\epsilon'$, a contradiction.
If $\varnothing\neq \tau\subseteq \tau'$, then in view of definition of $\Gamma$, we should have $l(\tau)=l(\tau')$.
We consider three different cases.\\
\begin{enumerate}[label={\rm (\roman*)}]
\item If $h(\tau)=h(\tau')=0$, then $$\epsilon=s_0(\{\epsilon\in \mathbb{Z}_p:\; 
         \tau^\epsilon=   \varnothing\})\neq s_0(\{\epsilon\in \mathbb{Z}_p:\;  
         {\tau'}^\epsilon=   \varnothing\})=\epsilon'.$$ 
          Therefore,   $ \{\epsilon\in \mathbb{Z}_p:\;  
         {\tau'}^\epsilon=   \varnothing\}\subsetneq\{\epsilon\in \mathbb{Z}_p:\; 
         \tau^\epsilon=   \varnothing\}$. This implies that 
         $$l(\tau')=p-|\{\epsilon\in \mathbb{Z}_p:\; 
         {\tau'}^\epsilon=   \varnothing\}|>p-|\{\epsilon\in \mathbb{Z}_p:\; 
         \tau^\epsilon=   \varnothing\}|=l(\tau),$$ 
         a contradiction.\\

\item If $h(\tau)=0$ and $h(\tau')>0.$ We should have 
	$l(\tau)\leq p-1$ and $l(\tau')\geq p$ which contradicts the fact that $l(\tau)=l(\tau')$.\\

\item If $h(\tau)>0$ and $h(\tau')>0.$ 
	 Note that 
	 $$ l(\tau)=p\cdot h(\tau)+|\{\epsilon\in\mathbb{Z}_p:\; |\tau^\epsilon|>h(\tau)\}|
	 \mbox{ and } l(\tau')=p\cdot h(\tau')+|\{\epsilon\in\mathbb{Z}_p:\; 
	 |{\tau'}^\epsilon|>h(\tau')\}|.$$
For this case, two different sub-cases will be distinguished.
	 \begin{itemize}
	 \item[(a)] If $h(\tau)=h(\tau')=h$, then 
	 $$\epsilon=s(\displaystyle\bigcup_{\{\epsilon\in\mathbb{Z}_p:\; |\tau^\epsilon|=h\}} \tau^\epsilon)\neq s(\displaystyle\bigcup_{\{\epsilon\in\mathbb{Z}_p:\; |{\tau'}^\epsilon|=h\}} {\tau'}^\epsilon)=\epsilon'.$$
	 Clearly, it implies that $$\displaystyle\bigcup_{\{\epsilon\in\mathbb{Z}_p:\; |\tau^\epsilon|=h\}} \tau^\epsilon\neq \displaystyle\bigcup_{\{\epsilon\in\mathbb{Z}_p:\; |{\tau'}^\epsilon|=h\}} {\tau'}^\epsilon.$$
	 Note that $\tau\subseteq \tau'$ and $h=\ds\min_{\epsilon\in \mathbb{Z}_p}|\tau^\epsilon|=\ds\min_{\epsilon\in \mathbb{Z}_p}|{\tau'}^\epsilon|.$ Therefore, we should have
	 $$
	 \{\epsilon\in\mathbb{Z}_p:\; |{\tau'}^\epsilon|=h\} \subsetneq \{\epsilon\in\mathbb{Z}_p:\; |{\tau}^\epsilon|=h\}$$
	 and consequently $l(\tau)<l(\tau')$ which is a contradiction.
	 \item[(b)] If $h(\tau)<h(\tau')$, then 
	 $$l(\tau)\leq p\cdot h(\tau)+p-1< p\cdot (h(\tau)+1)\leq l(\tau'),$$
	 a contradiction.	
	 \end{itemize}
\end{enumerate}
Therefore, $\Gamma$ is a simplicial $\mathbb{Z}_p$-map from $\sd K$ to $\mathbb{Z}_p^{*(n-1)}$. Naturally, $\lambda$ can be lifted to a simplicial $\mathbb{Z}_p$-map $\bar\lambda:\sd^2 \mathbb{Z}_p^{*n}\longrightarrow \sd K$. 
Thus $\Gamma\circ\bar\lambda$ is a  simplicial $\mathbb{Z}_p$-map from $\sd^2 \mathbb{Z}_p^{*n}$ to
$\mathbb{Z}_p^{*(n-1)}$.
In view of Dold's theorem~\cite{MR711043,MR1988723},  
the dimension of $\mathbb{Z}_p^{*(n-1)}$ should be strictly larger than the connectivity of $\sd^2 \mathbb{Z}_p^{*n}$, that is 
$n-2>n-2$,
which is not possible.
\end{proof}

Lemma~\ref{Z_pfanlemma} provides a short simple proof of Meunier's colorful result for Kneser hypergraphs (next Theorem) as follows.
\begin{alphtheorem}{\rm \cite{Meunier14}}\label{colorfulhyper}
Let ${\mathcal H}$ be a hypergraph and let $p$ be a prime number. Then any proper coloring $c:V({\rm KG}^p({\mathcal H}))\longrightarrow [C]$ {\rm(}$C$ arbitrary{\rm)} must contain a colorful balanced complete $p$-uniform $p$-partite hypergraph with 
$|V({\mathcal H})|-{\rm alt}_p({\mathcal H})$ vertices.
\end{alphtheorem}
\begin{proof}
Consider a bijection $\pi:[n]\longrightarrow V({\mathcal H})$ such that
${\rm alt}_p({\mathcal H},\pi)={\rm alt}_p({\mathcal H}).$
We are going to define a map $$\begin{array}{cccc}
\lambda: & (\mathbb{Z}_p\cup\{0\})^n\setminus\{\zero\} &\longrightarrow & \mathbb{Z}_p\times[m]\\
& X &\longmapsto & (\lambda_1(X),\lambda_2(X))
\end{array}$$ satisfying the conditions of Lemma~\ref{Z_pfanlemma}
and with parameters $n= |V({\mathcal H})|$, $m={\rm alt}_p({\mathcal H})+C$,
and $\alpha={\rm alt}_p({\mathcal H})$.
Assume that $2^{[n]}$ is equipped with a total ordering $\preceq$.
For each $X\in(\mathbb{Z}_p\cup\{0\})^n\setminus\{\zero\}$, define $\lambda(X)$ as follows.
\begin{itemize}
\item If ${\rm alt}(X)\leq {\rm alt}_p({\mathcal H},\pi)$, then let $\lambda_1(X)$ be the first nonzero coordinate of $X$ and $\lambda_2(X)={\rm alt}(X)$.
\item If ${\rm alt}(X)\geq {\rm alt}_p({\mathcal H},\pi)+1$, then in view of the definition of 
${\rm alt}_p({\mathcal H},\pi)$, there is some $\epsilon\in\mathbb{Z}_p$
such that $E(\pi(X^\epsilon))\neq \varnothing$.
Define 
$$c(X)=\max\left\{c(e):\; \exists\epsilon\in\mathbb{Z}_p\mbox { such that } 
e\subseteq \pi(X^\epsilon)\right\}$$
and $\lambda_2(X)={\rm alt}_p({\mathcal H},\pi)+c(X)$. 
Choose $\epsilon\in\mathbb{Z}_p$ such that 
there is at least one edge $e\in\pi (X^\epsilon)$ with $c(X)=c(e)$ and such that
 $X^\epsilon$ is the maximum one  having this property. By the maximum, we mean
 the maximum according to the total ordering $\preceq$. 
It is clear that $\epsilon$ is defined uniquely. Now, let $\lambda_1(X)=\epsilon$.
\end{itemize}
One can check that $\lambda$ satisfies the conditions of Lemma~\ref{Z_pfanlemma}.
Consider the chain $Z_1\subset Z_2\subset\cdots\subset Z_{n-{\rm alt}_p({\mathcal H},\pi)}$ whose existence is ensured by Lemma~\ref{Z_pfanlemma}.
Note that for each $i\in[n-{\rm alt}_p({\mathcal H},\pi)]$, we have $\lambda_2(Z_i)>{\rm alt}_p({\mathcal H},\pi)$. Consequently, $\lambda_2(Z_i)={\rm alt}_p({\mathcal H},\pi)+c(Z_i)$. Let $\lambda(Z_i)=(\epsilon_i,j_i)$. 
Note that for each $i$, there is at least one edge 
$e_{i,\epsilon_i}\subseteq \pi(Z_i^{\epsilon_i})\subseteq \pi(Z_{n-{\rm alt}_p({\mathcal H},\pi)}^{\epsilon_i})$ such that
$c(e_{i,\epsilon_i})=j_i-{\rm alt}_p({\mathcal H},\pi)$.
For each $\epsilon\in\mathbb{Z}_p$, define 
$U_\epsilon=\{e_{i,\epsilon_i}:\; \epsilon_i=\epsilon\}.$ 
We have the following three properties for $U_\epsilon$'s.
\begin{itemize}
\item Since the chain $Z_1\subset Z_2\subset\cdots\subset 
	 Z_{n-{\rm alt}_p({\mathcal H},\pi)}$ is satisfying Condition~\ref{condition3} of 
	 Lemma~\ref{Z_pfanlemma}, we have 
	 $\left\lfloor{n-{\rm alt}_p({\mathcal H},\pi)\over p}\right\rfloor\leq 
	 |U_\epsilon|\leq \left\lceil{n-{\rm alt}_p({\mathcal H},\pi)\over p}\right\rceil.$
\item The edges in $U_\epsilon$ get distinct colors. 
	 If there are two edges $e_{i,\epsilon}$ and $e_{i',\epsilon}$ in $U_\epsilon$ such that
	 $c(e_{i,\epsilon})=c(e_{i',\epsilon})$, then $\lambda(Z_i)=\lambda(Z_{i'})$ 
	 which is not possible.
\item If $\epsilon\neq \epsilon'$, then for each $e\in U_\epsilon$ and $f\in U_{\epsilon'}$, 
	 we have $e\cap f=\varnothing$. It is clear because 
	 $e\subseteq\pi(Z_{n-{\rm alt}_p({\mathcal H},\pi)}^\epsilon)$,
	 $f\subseteq\pi(Z_{n-{\rm alt}_p({\mathcal H},\pi)}^{\epsilon'})$,
	 and 
	 $$\pi(Z_{n-{\rm alt}_p({\mathcal H},\pi)}^\epsilon)\cap 
	 \pi(Z_{n-{\rm alt}_p({\mathcal H},\pi)}^{\epsilon'})=\varnothing.$$
\end{itemize}
Now, it is clear that the subhypergraph ${\rm KG}^p({\mathcal H})[U_{\omega^1},\ldots,U_{\omega^p}]$ is the desired subhypergraph.
\end{proof}


The proof of next lemma is similar to the proof of Lemma~\ref{Z_pfanlemma}.
\begin{lemma}\label{genfanlemma}
Let $C$ be a free  simplicial $\mathbb{Z}_p$-complex 
such that ${\rm ind}_{\mathbb{Z}_p}(C)\geq t$
and let $\lambda:C\longrightarrow (\sigma^{p-1}_{p-2})^{*m}$ be a simplicial $\mathbb{Z}_p$-map.
Then there is at least one $t$-dimensional simplex $\sigma\in C$ such that $\tau=\lambda(\sigma)$ is a $t$-dimensional simplex and for each $\epsilon\in \mathbb{Z}_p$, we have
$\lfloor{t+1\over p}\rfloor\leq |\tau^\epsilon|\leq\lceil{t+1\over p}\rceil.$
\end{lemma}
\begin{proof}
For simplicity of notation, let $K={\rm Im}(\lambda)$.
Clearly, to prove the assertion, it is enough to show that there is a $t$-dimensional simplex $\tau\in K$ 
such that $l(\tau)\geq t$.
Suppose, contrary to the assertion, that there is no such a $t$-dimensional simplex.
Therefore, for each simplex $\tau$ of $K$, we have $l(\tau)\leq t$.
For each vertex $\tau\in V(\sd K)$, set $h(\tau)=\ds\min_{\epsilon\in \mathbb{Z}_p}|\tau^\epsilon|$.

Let $\Gamma:\sd K\longrightarrow \mathbb{Z}_p^{*t}$  be a map such that for each vertex $\tau$ of $\sd K$, $\Gamma(\tau)$ is defined as follows.
\begin{enumerate}[label={\rm (\roman*)}]
\item   If $h(\tau)=0$, then define $\bar{\tau}=\{\epsilon\in \mathbb{Z}_p:\; 
                      \tau^\epsilon=   \varnothing\}\in \sigma^{p-1}_{p-2}$ and 
                      $$\Gamma(\sigma\cup\tau)=\left(s_0(\bar\tau), l(\tau)\right).$$
              
\item  If $h(\tau)> 0$, then define $\bar{\tau}=\displaystyle
	        \bigcup_{\{\epsilon\in\mathbb{Z}_p:\; |\tau^\epsilon|=h(\tau)\}} \tau^\epsilon\in W$ 
	        and $$\Gamma(\sigma\cup\tau)=\left(s(\bar\tau), l(\tau)\right).$$
\end{enumerate}
Similar to the proof of Lemma~\ref{Z_pfanlemma}, 
$\Gamma\circ\bar{\lambda}:\sd C\longrightarrow \mathbb{Z}_p^{*t}$ is a  simplicial $\mathbb{Z}_p$-map. This implies that ${\rm ind}_{\mathbb{Z}_p}(C)\leq t-1$ which is~not possible.
\end{proof}

Next proposition is an extension of Theorem~\ref{altercrossindex}. However, we lose some properties by this extension.

\begin{proposition}\label{Xindposet}
Let $P$ be a free ${\mathbb Z}_p$-poset and 
$$\begin{array}{rll}
\psi: P & \longrightarrow & Q_{s,p}\\
       p  &\longmapsto & (\psi_1(p),\psi_2(p))
\end{array}$$ 
be an order preserving ${\mathbb Z}_p$-map. Then $P$ contains a chain $p_1\prec_P\cdots\prec_Pp_{ k}$ such that 
\begin{itemize}
\item $k= {\rm ind}_{\Z_p}(\Delta P)+1$,
\item for each $i\in[k-1]$, $\psi_2(p_i)< \psi_2(p_{i+1})$, and
\item for each $\epsilon\in\mathbb{Z}_p$, 
$$\left\lfloor{k\over p}\right\rfloor\leq \left|\left\{j:\; \psi_1(p_j)=\epsilon\right\}\right|\leq \left\lceil{k\over p}\right\rceil.$$
\end{itemize}
\end{proposition}
\begin{proof}
Note $\psi$ can be considered as a simplicial ${\mathbb Z}_p$-map from $\Delta P$ to $ \mathbb{Z}_p^{*n}\subseteq (\sigma_{p-2}^{p-1})^{*n}$. Now, in view of Lemma~\ref{genfanlemma},
we have the assertion.
\end{proof}
Note that, for $p=2$,
since ${\rm Xind}(P)\geq {\rm ind}(\Delta P)$,
 Theorem~\ref{altercrossindex} is better than proposition~\ref{Xindposet}.
However, we cannot prove that proposition~\ref{Xindposet} 
is valid if we replace ${\rm ind}(\Delta P)$ by ${\rm Xind}(P)$.

In an unpublished paper, Meunier~\cite{unpublishedMeunier} introduced a 
generalization of Tuckey-Ky~Fan lemma. He presented a version of 
$\mathbb{Z}_q$-Fan lemma which is valid for each odd integer 
$q\geq 3$. To be more specific, he proved that if $q$ is an odd positive integer and 
$\lambda:V(T)\longrightarrow \Z_q\times[m]$ is a $\Z_q$-equivariant  labeling 
of an $\Z_q$-equivariant 
triangulation of a $(d-1)$-connected free $\Z_q$-spaces $T$, then 
there is at least one simplex in 
$T$ whose vertices are labelled with labels 
$(\epsilon_0,j_0),(\epsilon_1,j_1),\ldots,(\epsilon_n,j_n)$, 
where $\epsilon_i\neq \epsilon_{i+1}$ and $j_i<j_{i+1}$ for 
all $i\in\{0,1,\ldots,n-1\}$. Also, he asked the question if the result is true for
even value of $q$. This question received a positive answer owing 
to the work of B.~Hanke et~al.~\cite{Hanke2009404}.  
In both mentioned works, the proofs of $\mathbb{Z}_q$-Fan lemma 
are built in involved construction. Here, we take the opportunity of this paper to 
propose the following generalization of this result with a short simple 
proof because we are using similar techniques in the paper.
\begin{lemma}{\rm($\mathbb{G}$-Fan lemma).}\label{Gtucker}
Let $\mathbb{G}$ be a nontrivial finite group and 
let $T$ be a free $\mathbb{G}$-simplicial complex such that ${\rm ind}_{\mathbb{G}}(T)= n$.
Assume that $\lambda:V(T)\longrightarrow \mathbb{G}\times[m]$ be a $\mathbb{G}$-equivariant  labeling such that there is no edge in $T$ whose vertices are labelled  with $(g,j)$ and $(g',j)$ with $g\neq g'$ and $j\in[m]$. Then there is at least one simplicial complex in $T$ whose vertices are labelled with labels $(g_0,j_0),(g_1,j_1),\ldots,(g_n,j_n)$, where $g_i\neq g_{i+1}$ and $j_i<j_{i+1}$ for all $i\in\{0,1,\ldots,n-1\}$. In particular, $m\geq n+1$.
\end{lemma}
\begin{proof}
Clearly, the map $\lambda$ can be considered as a $\mathbb{G}$-simplicial map from $T$ to $\mathbb{G}^{*m}$. Naturally, 
each nonempty simplex $\sigma\in \mathbb{G}^{*m}$ can be identified with a vector $X=(x_1,x_2,\ldots,x_m)\in (\mathbb{G}\cup\{0\})^n\setminus\{\zero\}$. To prove the assertion, it is enough to show that there is a simplex $\sigma\in T$ such that ${\rm alt}(\lambda(\sigma))\geq n+1$. 
For a contradiction, suppose that, for each simplex $\sigma\in T$, we have ${\rm alt}(\lambda(\sigma))\leq n$. 
Define 
$$\begin{array}{lrll}
\Gamma:&V(\sd T) &\longrightarrow & \mathbb{G}\times[n]\\
		&\sigma&\longmapsto & \left(g,{\rm alt}(\lambda(\sigma)\right)),
\end{array}$$
where $g$ is the first nonzero coordinate of the vector $\lambda(\sigma)\in (\mathbb{G}\cup\{0\})^n\setminus\{\zero\}.$
One can check that $\Gamma$ is a  simplicial $\mathbb{G}$-map from $\sd T$ to $\mathbb{G}^{*n}$.
Note $\mathbb{G}^{*n}$ is an $E_{n-1} \mathbb{G}$ space. 
Consequently, ${\rm ind}_{\mathbb{G}}(\mathbb{G}^{*n})= n-1$.
This implies that ${\rm ind}_{\mathbb{G}}(T)\leq n-1$ which is a contradiction.
\end{proof}

\subsection{\bf Hierarchy of Indices}
The aim of this subsection is introducing some tools for the proof of Theorem~\ref{inequalities}.

Let $n,\alpha$, and $p$ be integers where $n\geq 1$, $n\geq\alpha\geq 0$, and $p$ is prime.
Define 
$$\displaystyle\Sigma_p(n,\alpha)=\Delta\left\{X\in(\mathbb{Z}_p\cup\{0\})^n:\; {\rm alt}(X)\geq \alpha+1\right\}.$$
Note that $\displaystyle\Sigma_p(n,\alpha)$ is a free simplicial $\mathbb{Z}_p$-complex
with the vertex set $$\left\{X\in(\mathbb{Z}_p\cup\{0\})^n:\; {\rm alt}(X)\geq \alpha+1\right\}.$$

\begin{lemma}\label{indsigma}
Let $n,\alpha$, and $p$ be integers where $n\geq 1$, $n\geq\alpha\geq 0$, and $p$ is prime. Then 
$${\rm ind}_{\mathbb{Z}_p}(\displaystyle\Sigma_p(n,\alpha))\geq n-\alpha-1.$$
\end{lemma}
\begin{proof}
Define
$$
\begin{array}{crcl}
\lambda:  & \sd \mathbb{Z}_p^{*n} & \longrightarrow & 
			(\mathbb{Z}_p^{*\alpha})*\left(\displaystyle\Sigma_p(n,\alpha)\right)\\
		& X					  & \longmapsto	     & 
		\left
			\{\begin{array}{cl}
			(\epsilon,{\rm alt}(X)) &  \mbox{ if ${\rm alt}(X)\leq \alpha$}\\
			                            X    & \mbox{ if ${\rm alt}(X)\geq \alpha+1$},
			\end{array}
		\right.
\end{array}$$
where $\epsilon$ is the first nonzero term of $X$.
Clearly, the map $\lambda$ is a simplicial $\mathbb{Z}_p$-map.
Therefore, 
$$
\begin{array}{lll}
n-1={\rm ind}_{\mathbb{Z}_p}( \sd \mathbb{Z}_p^{*n}) & \leq & {\rm ind}_{\mathbb{Z}_p}\left(\mathbb{Z}_p^{*\alpha}*\displaystyle\Sigma_p(n,\alpha)\right)\\
& \leq & {\rm ind}_{\mathbb{Z}_p}(\mathbb{Z}_p^{*\alpha})+{\rm ind}_{\mathbb{Z}_p}(\displaystyle\Sigma_p(n,\alpha))+1\\
&\leq &\alpha+{\rm ind}_{\mathbb{Z}_p}(\displaystyle\Sigma_p(n,\alpha))
\end{array}
$$ which completes the proof.
\end{proof}
\begin{proposition}\label{inequalityI}
Let ${\mathcal H}$ be a hypergraph. For any integer $r\geq 2$ and any prime number $p\geq r$, we have
$${\rm ind}_{\mathbb{Z}_p}({\rm B}_0({\rm KG}^r({\mathcal H}),\mathbb{Z}_p))+1\geq |V({\mathcal H})|-{\rm alt}_p({\mathcal H}).$$
\end{proposition}
\begin{proof}
For convenience, let $|V({\mathcal H})|=n$ and $\alpha=n-{\rm alt}_p({\mathcal H})$.
Let $\pi:[n]\longrightarrow V({\mathcal H})$ be the bijection such that 
${\rm alt}_p({\mathcal H},\pi)={\rm alt}_p({\mathcal H})$. 
Define 
$$
\begin{array}{lrll}
\lambda:& \Sigma_p(n,\alpha)& \longrightarrow & \sd{\rm B}_0({\rm KG}^r({\mathcal H}),\mathbb{Z}_p))\\
 & X&\longmapsto & \{\omega^1\}\times U_1\cup\cdots\cup \{\omega^p\}\times U_p,
\end{array}
$$  where $U_i=\{e\in E({\mathcal H}):\; e\subseteq \pi(X^{\omega^i})\}.$
One can see that $\lambda$ is a  simplicial $\mathbb{Z}_p$-map.
Consequently, 
$${\rm ind}_{\mathbb{Z}_p}({\rm B}_0({\rm KG}^r({\mathcal H}),\mathbb{Z}_p))\geq 
{\rm ind}_{\mathbb{Z}_p}(\Sigma_p(n,\alpha))\geq n-{\rm alt}_p({\mathcal H})-1.
$$
\end{proof}

\begin{proposition}\label{inequalityII}
Let ${\mathcal H}$ be an $r$-uniform hypergraph and $p\geq r$ be a prime number.
Then
$$ {\rm Xind}_{\mathbb{Z}_p}({\rm Hom}(K^r_p,{\mathcal H}))+p\geq  {\rm ind}_{\mathbb{Z}_p}(\Delta{\rm Hom}(K^r_p,{\mathcal H}))+p\geq  {\rm ind}_{\mathbb{Z}_p}(B_0({\mathcal H},\mathbb{Z}_p))+1.$$
\end{proposition}
\begin{proof}
Since already we know 
${\rm Xind}_{\mathbb{Z}_p}({\rm Hom}(K^r_p,{\mathcal H}))\geq  {\rm ind}_{\mathbb{Z}_p}(\Delta{\rm Hom}(K^r_p,{\mathcal H}))$, 
to prove the assertion, it is enough to show that 
${\rm ind}_{\mathbb{Z}_p}(\Delta{\rm Hom}(K^r_p,{\mathcal H}))+p\geq  {\rm ind}_{\mathbb{Z}_p}(B_0({\mathcal H},\mathbb{Z}_p))+1.$
To this end, define 
$$\begin{array}{llll}
\lambda: & \sd B_0({\mathcal H},\mathbb{Z}_p) & \longrightarrow & \left(\sd\sigma_{p-2}^{p-1}\right)*\displaystyle\left(\Delta{\rm Hom}(K^r_p,{\mathcal H})\right)\\
\end{array}$$
such that for each vertex 
$\tau=\displaystyle\bigcup_{i-1}^p\left(\{\omega^i\}\times U_i\right)$ of $\sd B_0({\mathcal H},\mathbb{Z}_p)$, $\lambda(\tau)$ is defined as follows.
\begin{itemize}
\item If $U_i\neq\varnothing$ for each $i\in[p]$, then $\lambda(\tau)=\tau.$
\item If $U_i=\varnothing$ for some $i\in[p]$, then 
$$\lambda(\tau)=\{\omega^i\in\mathbb{Z}_p:\; U_i=\varnothing\}.$$
\end{itemize} 
One can check that the map $\lambda$ is a simplicial $\mathbb{Z}_p$-map. Also,
since $\sigma_{p-2}^{p-1}$ is a free simplicial $\mathbb{Z}_p$-complex of 
dimension $p-2$, we have ${\rm ind}_{\mathbb{Z}_p}(\sigma_{p-2}^{p-1})\leq p-2$
(see properties of the $\mathbb{G}$-index in Section~\ref{intro}).
This implies that 
$$
\begin{array}{lll}
 {\rm ind}_{\mathbb{Z}_p}(B_0({\mathcal H},\mathbb{Z}_p))& \leq &  {\rm ind}_{\mathbb{Z}_p}\left(\left(\sd\sigma_{p-2}^{p-1}\right)*
\left(\Delta{\rm Hom}(K^r_p,{\mathcal H})\right)\right)\\
& \leq &{\rm ind}_{\mathbb{Z}_p}(\sigma_{p-2}^{p-1})+ {\rm ind}_{\mathbb{Z}_p}(\Delta{\rm Hom}(K^r_p,{\mathcal H}))+1\\
&\leq & p-1+{\rm ind}_{\mathbb{Z}_p}(\Delta{\rm Hom}(K^r_p,{\mathcal H}))
\end{array}
$$ which completes the proof.
\end{proof}
\section{\bf Proofs of Theorem~\ref{maincolorfulindex} and Theorem~\ref{inequalities}}\label{sec:proofs}
Now, we are ready to prove Theorem~\ref{maincolorfulindex} and Theorem~\ref{inequalities}.\\

\noindent{\bf Proof of Theorem~\ref{maincolorfulindex}: Part (i).}
For convenience, let ${\rm ind}_{\mathbb{Z}_p}({\rm B}_0({\mathcal H},{\mathbb{Z}_p}))=t$.
Note that
$$
\begin{array}{crcl}
\Gamma:  &\mathbb{Z}_p\times V({\mathcal H}) & \longrightarrow & \mathbb{Z}_p\times [C]\\
		  & (\epsilon,v)		  			   & \longmapsto	     & (\epsilon,c(v))
\end{array}$$ is a simplicial $\mathbb{Z}_p$-map from ${\rm B}_0({\mathcal H},{\mathbb{Z}_p})$ to $(\sigma^{p-1}_{r-2})^{*C}$. Therefore, in view of Lemma~\ref{genfanlemma}, there is a $t$-dimensional simplex $\tau\in{\rm im}(\Gamma)$ such that, for each $\epsilon\in \mathbb{Z}_p$, we have
$\lfloor{t+1\over p}\rfloor\leq |\tau^\epsilon|\leq\lceil{t+1\over p}\rceil.$
Let $\ds\bigcup_{i=1}^p(\{\omega^i\} \times U_i)$ be the minimal simplex in $\Gamma^{-1}(\tau)$. 
One can see that ${\mathcal H}[U_1,\ldots,U_p]$ is the desired subhypergraph.
Moreover,
since every color can be appeared in at most $r-1$ number of $U_i$'s,  we have
$$C\geq {{\rm ind}_{p}({\rm B}_0({\mathcal H},{\mathbb{Z}_p}))+1\over r-1}.$$

\noindent{\bf Part (ii).}
For convenience, let ${\rm Xind}_{\mathbb{Z}_p}({\rm Hom}(K^r_p,{\mathcal H}))=t$.
Define the map $$\lambda:{\rm Hom}(K^r_p,{\mathcal H})\longrightarrow \sd(\sigma_{r-2}^{p-1})^{*C}$$
such that for each $(U_1,\cdots,U_p)\in {\rm Hom}(K^r_p,{\mathcal H})$, 
$$\lambda(U_1,\cdots,U_p)=\{\omega^1\}\times c(U_1) \cup\cdots\cup \{\omega^p\}\times c(U_p).$$
{\bf Claim.} There is a $p$-tuple $(U_1,\cdots,U_p)\in {\rm Hom}(K^r_p,{\mathcal H})$
such that for $\tau=\lambda(U_1,\cdots,U_p)$, we have $l(\tau)\geq {\rm Xind}_{\mathbb{Z}_p}({\rm Hom}(K^r_p,{\mathcal H}))+p$.

\noindent{\bf Proof of Claim.} Suppose, contrary to the claim, that for each $\tau\in {\rm Im}(\lambda)$, we have $l(\tau)\leq t+p-1$.
Note that $\sd(\sigma_{r-2}^{p-1})^{*C}$ can be considered as a free $\mathbb{Z}_p$-poset ordered by inclusion.
One can readily check that $\lambda$ is an  order-preserving $\mathbb{Z}_p$-map.
Clearly, for each $\tau\in {\rm Im}(\lambda)$, we have $h(\tau)=\ds\min_{\epsilon\in\mathbb{Z}_p}|\tau^\epsilon|\geq 1$ and consequently, $l(\tau)\geq p$. Now, define  $$\bar{\tau}=\displaystyle
	        \bigcup_{\{\epsilon\in\mathbb{Z}_p:\; |\tau^\epsilon|=h(\tau)\}} \tau^\epsilon\in W\quad {\rm and }\quad \Gamma(\tau)=\left(s(\bar\tau), l(\tau)-p+1\right).$$
One can see that the map $\Gamma:{\rm im}(\lambda)\longrightarrow Q_{t-1,p}$ is an order-preserving $\mathbb{Z}_p$-map. Therefore, 
$$\Gamma\circ\lambda:{\rm Hom}(K^r_p,{\mathcal H})\longrightarrow Q_{t-1,p}$$ is an 
order-preserving $\mathbb{Z}_p$-map, which contradicts the fact that ${\rm Xind}_{\mathbb{Z}_p}({\rm Hom}(K^r_p,{\mathcal H}))=t$.\hfill$\square$

Now, let $(U_1,\cdots,U_p)$ be a minimal $p$-tuple in ${\rm Hom}(K^r_p,{\mathcal H})$ 
such that for $$\tau=\lambda(U_1,\cdots,U_p)=\{\omega^1\}\times c(U_1) \cup\cdots\cup \{\omega^p\}\times c(U_p),$$ we have $l(\tau)= t+p$.
One can check that ${\mathcal H}[U_1,\cdots,U_p]$ is the desired complete 
$r$-uniform $p$-partite subhypergraph. Similar to the proof of Part (i),
since every color can be appeared in at most $r-1$ number of $U_i$'s,  we have 
$$C\geq {{\rm Xind}_{\mathbb{Z}_p}({\rm Hom}(K^r_p,{\mathcal H}))+p\over r-1}.$$
\hfill$\square$

\noindent{\bf Proof of Theorem~\ref{inequalities}.}
It is simple to prove that
$|V({\mathcal F})|-{\rm alt}_p({\mathcal F})
\geq  {\rm cd}_p({\mathcal F})$ for any hypergraph ${\mathcal F}$.
Therefore, the proof follows by  Proposition~\ref{inequalityI} and Proposition~\ref{inequalityII}.
\hfill$\square$\\

\noindent{\bf Acknowledgements.}
I would like to acknowledge Professor Fr\'ed\'eric~Meunier for interesting discussions about the paper and his invaluable comments. 
Also, I would like to thank Professor Hossein~Hajiabolhasan and Mrs~Roya~Abyazi~Sani for their useful comments.  


\begin{thebibliography}{10}

\bibitem{2014arXiv1401.0138A}
M.~{Alishahi} and H.~{Hajiabolhassan}.
\newblock {Chromatic Number Via Turan Number}.
\newblock {\em ArXiv e-prints}, arXiv:1401.0138v4, December 2014.

\bibitem{2014arXiv1403.4404A}
M.~{Alishahi} and H.~{Hajiabolhassan}.
\newblock {Hedetniemi's Conjecture Via Altermatic Number}.
\newblock {\em ArXiv e-prints}, arXiv:1403.4404v4, March 2014.

\bibitem{2014arXiv1407.8035A}
M.~{Alishahi} and H.~{Hajiabolhassan}.
\newblock {On Chromatic Number and Minimum Cut}.
\newblock {\em ArXiv e-prints}, arXiv:1407.8035v2, July 2014.

\bibitem{2015arXiv150708456A}
M.~{Alishahi} and H.~{Hajiabolhassan}.
\newblock {On The Chromatic Number of Matching Graphs}.
\newblock {\em ArXiv e-prints}, arXiv:1507.08456v1, July 2015.

\bibitem{AliHajiMeu2016}
M.~{Alishahi}, H.~{Hajiabolhassan}, and F.~{Meunier}.
\newblock {Strengthening topological colorful results for graphs}.
\newblock {\em ArXiv e-prints}, arXiv:1606.02544, June 2016.

\bibitem{2013arXiv1302.5394A}
M.~Alishahi and H.~Hajiabolhassan.
\newblock On the chromatic number of general kneser hypergraphs.
\newblock {\em Journal of Combinatorial Theory, Series B}, 115:186 -- 209,
  2015.

\bibitem{MR857448}
N.~Alon, P.~Frankl, and L.~Lov{\'a}sz.
\newblock The chromatic number of {K}neser hypergraphs.
\newblock {\em Trans. Amer. Math. Soc.}, 298(1):359--370, 1986.

\bibitem{MR2971704}
G.~J.~Chang, D.~D.-F.~Liu, and X.~Zhu.
\newblock A short proof for {C}hen's {A}lternative {K}neser {C}oloring {L}emma.
\newblock {\em J. Combin. Theory Ser. A}, 120(1):159--163, 2013.

\bibitem{MR2763055}
P.-A.~Chen.
\newblock A new coloring theorem of {K}neser graphs.
\newblock {\em J. Combin. Theory Ser. A}, 118(3):1062--1071, 2011.

\bibitem{MR711043}
A.~Dold.
\newblock Simple proofs of some {B}orsuk-{U}lam results.
\newblock In {\em Proceedings of the {N}orthwestern {H}omotopy {T}heory
  {C}onference ({E}vanston, {I}ll., 1982)}, volume~19 of {\em Contemp. Math.},
  pages 65--69. Amer. Math. Soc., Providence, RI, 1983.

\bibitem{MR953021}
V.~L. Dol{\cprime}nikov.
\newblock A combinatorial inequality.
\newblock {\em Sibirsk. Mat. Zh.}, 29(3):53--58, 219, 1988.

\bibitem{ERDOS198621}
P.~Erd{\H{o}}s, Z.~F{\"u}redi, A.~Hajnal, P.~Komj{\'a}th, V.~R{\"o}dl, and
  {\'A}.~Seress.
\newblock Coloring graphs with locally few colors.
\newblock {\em Discrete Math.}, 59(1-2):21--34, 1986.

\bibitem{MR0051506}
K.~Fan.
\newblock A generalization of {T}ucker's combinatorial lemma with topological
  applications.
\newblock {\em Ann. of Math. (2)}, 56:431--437, 1952.

\bibitem{HaMe16}
H.~Hajiabolhassan and F.~Meunier.
\newblock Hedetniemi's conjecture for kneser hypergraphs.
\newblock {\em Journal of Combinatorial Theory, Series A}, 143:42 -- 55, 2016.

\bibitem{MR2837625}
H.~Hajiabolhassan.
\newblock A generalization of {K}neser's conjecture.
\newblock {\em Discrete Math.}, 311(23-24):2663--2668, 2011.

\bibitem{Hanke2009404}
B.~Hanke, R.~Sanyal, C.~Schultz, and G.~M.~Ziegler.
\newblock Combinatorial stokes formulas via minimal resolutions.
\newblock {\em Journal of Combinatorial Theory, Series A}, 116(2):404 -- 420,
  2009.

\bibitem{Iriye20131333}
K.~Iriye and D.~Kishimoto.
\newblock Hom complexes and hypergraph colorings.
\newblock {\em Topology and its Applications}, 160(12):1333 -- 1344, 2013.

\bibitem{MR0068536}
M.~Kneser.
\newblock Ein {S}atz \"uber abelsche {G}ruppen mit {A}nwendungen auf die
  {G}eometrie der {Z}ahlen.
\newblock {\em Math. Z.}, 61:429--434, 1955.

\bibitem{MR1081939}
I.~K{\v{r}}{\'{\i}}{\v{z}}.
\newblock Equivariant cohomology and lower bounds for chromatic numbers.
\newblock {\em Trans. Amer. Math. Soc.}, 333(2):567--577, 1992.

\bibitem{MR514625}
L.~Lov{\'a}sz.
\newblock Kneser's conjecture, chromatic number, and homotopy.
\newblock {\em J. Combin. Theory Ser. A}, 25(3):319--324, 1978.

\bibitem{MR1988723}
J.~Matou{\v{s}}ek.
\newblock {\em Using the {B}orsuk-{U}lam theorem}.
\newblock Universitext. Springer-Verlag, Berlin, 2003.
\newblock Lectures on topological methods in combinatorics and geometry,
  Written in cooperation with Anders Bj{\"o}rner and G{\"u}nter M. Ziegler.

\bibitem{MR2057690}
J.~Matou{\v{s}}ek.
\newblock A combinatorial proof of {K}neser's conjecture.
\newblock {\em Combinatorica}, 24(1):163--170, 2004.

\bibitem{Meunier14}
F.~{Meunier}.
\newblock {Colorful Subhypergraphs in Kneser Hypergraphs}.
\newblock {\em Electron. J. Combin.}, 21(1):\ Research Paper \#P1.8, 13 pp.
  (electronic), 2014.

\bibitem{unpublishedMeunier}
F.~Meunier.
\newblock A $\mathbb{Z}_q$-fan theorem.
\newblock {\em presented at the ``Topological combinatorics workshop" in
  Stockholm, Sweden}, May 2006.

\bibitem{MR2793613}
F.~Meunier.
\newblock The chromatic number of almost stable {K}neser hypergraphs.
\newblock {\em J. Combin. Theory Ser. A}, 118(6):1820--1828, 2011.

\bibitem{SiTaZs13}
G.~Simonyi, C.~Tardif, and A.~Zsb{\'{a}}n.
\newblock Colourful theorems and indices of homomorphism complexes.
\newblock {\em Electr. J. Comb.}, 20(1):P10, 2013.

\bibitem{MR2279672}
G.~Simonyi and G.~Tardos.
\newblock Local chromatic number, {K}y {F}an's theorem and circular colorings.
\newblock {\em Combinatorica}, 26(5):587--626, 2006.

\bibitem{MR2351519}
G.~Simonyi and G.~Tardos.
\newblock Colorful subgraphs in {K}neser-like graphs.
\newblock {\em European J. Combin.}, 28(8):2188--2200, 2007.

\bibitem{MR0020254}
A.~W.~Tucker.
\newblock Some topological properties of disk and sphere.
\newblock In {\em Proc. {F}irst {C}anadian {M}ath. {C}ongress, {M}ontreal,
  1945}, pages 285--309. University of Toronto Press, Toronto, 1946.

\bibitem{MR1893009}
G.~M.~Ziegler.
\newblock Generalized {K}neser coloring theorems with combinatorial proofs.
\newblock {\em Invent. Math.}, 147(3):671--691, 2002.

\end{thebibliography}

\def\cprime{$'$} \def\cprime{$'$}

\end{document}